\documentclass{amsart}

\usepackage{amsmath,amssymb,enumerate, graphicx}

\usepackage[latin1]{inputenc}
\providecommand{\Phi}{\textbf{\Phi}}

\newcommand{\J}{\mathcal{J} }

\newcommand{\F}{\mathcal{F }}

\newcommand{\N}{\mathbb{N} }

\newcommand{\Erond}{\mathcal{E}}

\renewcommand{\H}{\mathcal{H}}

\newcommand{\M}{\mathcal{M}}

\renewcommand{\P}{\mathbb{P}}

\providecommand{\sup}{\mathrm{sup}}

\providecommand{\inf}{\mathrm{inf}}

\providecommand{\dim}{\mathrm{dim} \,}

\providecommand{\C}{\mathcal{C}} 

\renewcommand{\L}{\mathcal{L}}

\providecommand{\G}{\mathcal{G}} \providecommand{\A}{\mathcal{A}}

\providecommand{\R}{\mathbb{R}}

\providecommand{\F}{\mathbb{F}}

 \providecommand{\E}{\mathbb{E}}

\providecommand{\P}{\mathbb{P}} \providecommand{\N}{\mathbb{N}}

\providecommand{\Z}{\mathbb{Z}}
\providecommand{\Q}{\mathbb{Q}}

\newtheorem{theo}{Theorem}[section]

\newtheorem{lem}[theo]{Lemma}

\newtheorem{souslem}[theo]{Sublemma}

\newtheorem{prop}[theo]{Proposition}

\newtheorem{rem}[theo]{Remark}

\newtheorem{defi}[theo]{Definition}

\begin{document}

\title[ ]{\textbf{Existence
of graphs with sub exponential  transitions probability decay and
applications .}}

\author{Cl\'ement Rau}

\date{}

\begin{abstract}

 In this paper, we  present a complete proof of the
construction of  graphs with
bounded valency such that the simple random
walk has  a return probability  at time $n$ at the origin of order
$exp(-n^{\alpha}),$ for fixed $\alpha \in [0,1[$ and with Folner function $exp(n^{
\frac{2\alpha}{1-\alpha}})$.  We begin by giving a more detailled proof
of this result contained in  (see \cite{ershdur}).\\    
In the second part, we give an application of the existence of such graphs.
We obtain bounds of the correct order for some functional of the local
time of a simple random walk on an infinite cluster on the percolation
model.

\end{abstract}

\maketitle \tableofcontents

\section{Introduction and results}

 A graph $G$ is a couple  $(V(G),E(G))$, where  
 $V(G)$  stands for the set of  vertices  of $ G$  et
 $E(G)$  stands for the set of edges of $G$. All
 graphs
 $G$ which are considered here  are infinite and have bounded geometry
 and we denote by $\nu(g)$ the number of neighbors of $g$
 in $G$.\\
We study the following random walk $X$ on $G$ defined
by:
 \begin{eqnarray}\label{eq1}
 \left\lbrace
\begin{array}{l}
  X_0=g, \\
 \P(X_{n+1}=b|X_n=a) = \frac{1}{\nu(a) +1} (1_{\{(a,b)\in E(G) \}} + 
 1_{\{a=b \}} ) \\
\end{array}
\right.
\end{eqnarray}
The random walk $X$ jumps uniformly on the set of
points formed by  the point where the walker is and
his neighbors. Thus $X$ admits reversible measures which are
proportionnal to $m(x)=\nu(x) +1$.\\

In this context,  the  transition probabilties  are linked by the isoperimetric
profile. For a graph $G$ and for a
subset $A$ of $ G $, we introduce  the boundary of $A$ relatively to
graph $G$ defined by   $$ \partial_G A=\{ (x,y)\in E(G);  \ x \in A  \ \text{et} \ y\in V(G)-A    \}.$$
 Actually, we will rather work with  F\o lner function to deal
 with isoperimetry. Let  $G$ be a graph, we note  $Fol_G$
  the  F\o lner function of  $G$ defined by:
$$ Fol_G(k) = \min \{ |U|; \ U\subset V(G)  \ \text{et} \  
 \frac{|\partial_{G} U|}{ |U|}\leq \frac{1}{k} \}.$$
If $G'\subset G$ is a subgraph of $G$,  we will use the
Folner function of $G'$ relatively to $G$ defined
by:$$ Fol_{G'}^G(k) = \min \{ |U|; \ U\subset V(G)  \ \text{et} \  
 \frac{|\partial_{G} U|}{ |U|}\leq \frac{1}{k} \}.$$
We have the following proposition (see coulhon
\cite{coulhon})
 \begin{prop} \label{propfol} Let $m_0=\inf_{V(U)} \
 m>0$ and $X$ be the random walk defined by (\ref{eq1}).   
 Assume that  $Fol(n) \geq F(n)$ with  $F$ a  non negative and non
 decreasing function, then 
 $$\sup_{x,y} \P(X_{n}=y|X_0=x)   \preceq  v(n),$$ 
 where  $v$ satisfies :
 $$\left\lbrace
\begin{array}{l}
  v'(t)= -\frac{ v(t)}{8 (F^{-1}(4/v(t) ) )^2},\\
 v(0)=1/m_0.\\
\end{array}
\right. $$
 \end{prop}
(We recall  that $a_n\preceq b_n$ if there exists
constants $c_1$ and $c_2$ such that for all $n\geq 0,
\  a_n \leq c_1b_{c_2 n} $  and $a_n \approx b_n $ if $a_n\preceq b_n$
and $a_n\succeq b_n$.) \\
For example, we retrieve  that in $ \Z^d$, the random walk $X$
defined above has transitions decay at time $n$ less
than $n^{-d/2}$ and in   $\F_2$ the Cayley graph of
the free group with two elements, the transition decay
of the random walk are less than $e^{-n}$.
A natural question is to know if there  exists  graphs
with intermediate transitions decay.  Some  others motivations can be
found in section \ref{appli}.\\

From $\Z$, one can perhaps adjust  some weigths on
edges  to get
the expected decay  but we  look after a graph with no
weigths.Indeed,  there are combinatorics
arguments in section \ref{appli} that will not work
if any weigths are present.

Our main result is :
\begin{prop} \label{existence} Let  $\alpha \in [0;1[,$ 
$F:=e^{x^{\frac{2\alpha}{ 1-\alpha} }}$ and $\sigma(n):
=e^{-n^{\alpha}} .$ There  exists
a  graph $D_F=(V(D_F),E(D_F))$  with bounded valency
such that :\\
(i) $Fol_{D_F} \approx F$,\\
(ii) there exists a point  $d_0 \in V(D_F)$ such that, for all  $n,\ p^{D_F}_n(d_0,d_0) \approx \sigma(n)$,\\
 where  $p^{D_F}_n(\ ,\ )$ stands for the transitions
 probability of the random walk $X$ defined
 above when $G=D_F$. 

\end{prop}

\subsection{Example of application of proposition \ref{existence}.}
With the help of these graphs and with  some  good wreath products, we
will be able to find upper bound of functional of type:

$\E( e^{-\lambda \sum F(L_{x,n}, x) }  )$ where
$L_{x,n}=\#\{ k\in[0;n];\ X_k=x\}$ on the graph $\C^g$ get
after a surcritical percolation on edges of $\Z^d$, where edges are kept or
removed with respect Bernouilli independant variables. The points of
$\C^g $ are the point of the infinite connected
component $\C$ which contains the origin,  we will give more details in  section
\ref{appli}. In particular,  we will prove the following property: 

 \begin{theo}\label{ext} Consider a simple random walk $X$ on the
 infinite cluster of $\Z^d $ that contains the origin 
$ \ Q \ a.s$  on the set $|\C|=+\infty$, and for large
enough   $n$ we have: 
\begin{eqnarray} \label{ext1}\forall \alpha \in [0,1] \ \ \  \E_0^{\omega} (
e^{
- \lambda \underset{z; L_{z;n}>0  } {\sum} L_{z;n}^{\alpha} 
} \ 
 1_{ 
 \{    X_{n}=0  \} 
 } 
) \approx  e^{-n^{\eta}},   \end{eqnarray} 
\begin{eqnarray} \label{ext2}     \forall \alpha > 1/2 \ \
\E_0^{\omega}( \underset{ z;L_{z;n} >0}{\prod }  L_{z;n}^{-\alpha}  \ \ 1_{\{X_n=0\}} ) 
 \approx  e^{-n^{\frac{d}{d+2}} 
 ln(n)^{\frac{2}{d+2} }},\end{eqnarray}
where  $\eta= \frac{ d+\alpha (2-d) }{ 2+d(1-\alpha) }.$ \\
The  constants  present  in the relation  $\approx$ 
 do not depend on the cluster $ \omega.$
\end{theo}

\begin{rem} If we take $\alpha=0$ in equation
(\ref{ext1}), we retrieve the Laplace transform of the
number of visited points  $N_n$ (see \cite{KL}),
$$\E^{\omega}_0(e^{-\lambda N_n})\approx
e^{-n^{d/d+2} } .$$
\end{rem}

In the whole article, $C, c$ are constants which value can evolve
from lines to lines. 
\section{Proof of proposition  \ref{existence} }
In this section, we first  recall the definition of the
wreath product of two graphs and we explain  our
strategy  aimed at the  construction of  our expected graphs. This leads
naturally towards two cases  corresponding  to the two
last subsections.

\subsection{Wreath products and explanation  of our
method }
Let $A$ a graph  and $(B_z)_{z\in A}$   a family of  graphs.
\begin{defi} The wreath product of $A$ and
$(B_z)_{z\in A}$ is the
graph noted by $A\wr B_{z} $ such that:\\
$$V(A\wr B_z)= \{(a,f);\ a\in A \text{ and }
f:A\rightarrow \cup_z B_z \text{ with } supp(f)<\infty $$
$\hspace*{5.5cm} \text{ and } \ \ \forall z\in A, \ f(z)\in B_z \}$\\
and 
$\ \ \ \ \ \ \ E(A\wr B) = \{ \Big( (a,f)(b,g) \Big);\ 
(f=g \ \text{ and }\  (a,b)\in E(A)  ) $\\
$\hspace*{7cm}\text{or}$ \\
$\hspace*{3.5cm}(a=b \text{ and } \forall x\neq a \ f(x)=g(x) \text{ and }
\big(f(a),g(a)\big) \in E(B_a)  )\}$

\end{defi} 
This graph can be interpreted as follow:  imagine there is a lamp in each point $a$ of $A$ such that
each point of $B_a$  defined a different intensity  of the
lamp. The different intensity of
each lamp can be represented  by a configuration $f: A\rightarrow
\cup_a B_a$ which 
encodes the intensity of the lamp at point $a$ by the value
$f(a)$. A point in the wreath product is the  couple
formed by the position of a  walker in graph $A$ and the
state of each lamp.
A particular case is when  the graph
$B_a$ (called the fiber) is the same for all $a \in
A$.

 Let us now explain the way we  construct
 graph $D_F$ of proposition  \ref{existence}.
 Consider the wreath product of the Cayley  graph of
 $(\Z,+)$  by the Cayley graph of $\frac{\Z}{2\Z}$ 
 with $\bar{1}$ as generator. By the Theorem 1 in
 \cite{ersh}  (or Proposition 3.2.1 in
 \cite{rothese}) we immediatly deduce that the
 Folner fonction of this wreath product is like   $e^n$.
So this graph answers to proposition \ref{existence} in the case  $\frac{2\alpha}{1-\alpha}=1
$.  ie : $ \alpha =1/3$.\\
In the case  $\alpha \neq 1/3$,  it would be rather natural to 
think that we can get the expected  graph, 
by considering  the wreath product of $\Z$
by  fibers with variable sizes.

$\bullet$ If $\alpha \geq 1/3$,  the return 
probability  in the graph  $D_F $  should be
  in $e^{-n^{\alpha} } $ so less than in the
  graph  $\Z\wr \frac{\Z}{2\Z}$ ( in 
  $e^{-n^{1/3}}$) . Thus to force the walk
  to come back rarely at the origin, an
  idea is to make the size of the fibers  grow  when   we
  move away the origin in order to force
   the walk to loose  time  in the fiber. \\
 Note  that for  $\alpha \geq 1$  condition
 (ii) is always satisfied  (in a graph with
 bounded geometry). \\

$\bullet$ If $\alpha \leq 1/3$,  the return 
probability  in the graph  $D_F $  should be
  larger than  in   $e^{-n^{1/3}}$. The idea
  is to add some links (some edges by
  example) to force  the walk to come back
  often to the origin.  Suppose all lamps
  are identified then we get a decay in
  $n^{-1/2}$ and if all lamps are
  independent we get a decay in
  $e^{-n^{1/3}}$, so it remains to find an
  identification  of lamps which implies an intermediate decay.
 We are going to construct a wreath product
 where the walker (at a certain point)  is allowed to   change the value
 of the configuration at differents  points.
  Such graphs are sometimes called
  generalized wreath products. \\
  
  To prove  isoperimetric inequality on wreath
  product ( point (i)
  of the proposition \ref{existence})  we use idea of Erschler and
  the concept of  "satisfactory" points.   We  begin to introduce
  this notion in section \ref{f}. At the beginning of section
  \ref{d}, we explain    why  an improvement is needed in the definition
  of "satisfactory"
   points. The improvement takes place through the introduction of  a new
   and more theoretical way of defining the  notion of "satisfactory"
   points 
    than in
  section \ref{f}. For simplicity, we use the same words  for
  this concept in the two sections but notions which appear in 
  sections
  \ref{d} and \ref{f}  are independent.

 \subsection{case $\frac{1}{3} \leq \alpha <
 1$}\label{f} $\ $ \\
 
 \subsubsection{Construction of the graph and
 preliminary notions and lemmas}

Let $A'=(\Z, E(\Z))$  where  $E(\Z)=\{(x,y);
\ |x-y|=1\}$ and
 $(B'_z)_{z\in \Z}$  be  the  Cayley graph of
 the   groups $(\frac{ \Z}{l(z) \Z}, +)$
 with   $\{\bar{1}\}$ as  generators where   $l(z)=
 |V(B_z')|=\frac{F(|z|+1)}{F(|z|) },$ ($F$ is defined at proposition
 \ref{existence}).\\
Notice that since  $\alpha \in [1/3, 1],$ the
fonction  $ \ z \mapsto l(z)$ is increasing
on   $\R_+$.\\
Finally put $$D_F = A' \wr B'_z.$$  Let us
prove that this graph answers to propostion
\ref{existence}.
\\

We begin by proving  (i). \\
The proof is  similar to the Theorem 1 in
\cite{ersh} or proposition 3.2.1 in
\cite{rothese}.\\
\\
Let  $\psi(n)=Fol_{A'}( n) = \underset{
\underset{ \frac{|\partial_{A'} U|}{ |U|} \leq 1/n}{U\subset \Z} }{min} |U|
=2n $.\\
 
 Take  $U\subset V(D_F)=V(A'\wr B'_z)$ such
 that  $\frac{|\partial_{D_F} U|}{|U|}\leq
 1/n$ for some $n$. We want to find a lower bound on
 $|U|$.\\
 
$\bullet$ For each set $U$, we attach an hypergraph  $K_U= \Bigl( V(K_U), \xi( K_U ) \Bigr)$
such that:\\
- the  vertices of  $K_U $ are the
configurations $f$ which belong to  the set $\{f ; \exists a \in \Z \  (a,f)\in U \}$, \\
- let us now define the edges of   $K_U$ : 
for all  $f\in V(K_U)$ and  $a\in \Z$, we
link $f$ to all configurations g satisfying: 
$$   \left\lbrace
\begin{array}{l}
(a,g)  \in U \\
and \\
\forall x\neq a  \ f(x)=g(x),
\end{array}
\right.
$$ by a  multidimensional edge  $l$ of
dimension $d$ where  $$d=\underset{a}{dim} f := \#\{ g; (a,g) \in U \ and
 \ \forall x \neq a \ f(x)=g(x)  \}.$$
 We say that the edge $l$ is
 associated to point $a$. \\
\\
$\bullet$ To each   hypergraph $K_U $ we associate a 
graph called the  " one dimensional skeleton",  noted by  $\Gamma(
K_U)=\Gamma_U=(V(\Gamma_U), E(\Gamma_U) )$ and
 defined by: \\
 - $V(\Gamma_U)=V(K_U),$\\
 - two configurations $ f_1$ and  $f_2$ are
 linked by an edge if they belong to a same
 multidimensional edge in  $K_U$.  \\

Let  $w$ be  the weight defined by $w(e)= 1/d $
 for  $e$ belonging to $ E( \Gamma_U)$ and coming 
 from a  multidimensional edge  in $K_U$ of
 dimension $d$. Notice that  this choice of weights
 gives :
 \begin{eqnarray} \label{remcardinalpoids}   |U| \geq 2
 \sum_{e\in E(\Gamma_U)} w(e),
 \end{eqnarray}
 and if we assume  moreover that for all $(x,f)\in U,\ dim_x f \geq 1$ ($U $ has
 no separeted  points) then the equality holds in
 \ref{remcardinalpoids}
  Let   $p$   be the   projection $\Z\wr B'_z
\rightarrow \Z $.   Let us now introduce some
notations.  Denote   
$\lambda=(\lambda_a)_{a\in p(U)} \in \R^{p(U)}$ and  $b \geq 0$. \\
\\ 
 $\bullet$ For $f\in V(K_U)$, we say that  
 $f$ is 
$(\lambda,b)-satisfactory $  
if :
  $$\#\{a \in p(V) ; \underset{a}{dim} f \geq \lambda_a      \} \geq b.
  $$

ie : $f$ is $(\lambda,b)-satisfactory$  if
there exists at least  $b$ multidimensional
edges  attached to  $f$  in $K_U$ of
dimension  at least $\lambda_a$ at point $a$.
We  denote by $S_U(\lambda,b)$  the set of these
points. Most of the time,  in order to simplify notations
we will drop the subscript $U$ when there is no
ambiguity.  \\
$\bullet$ Otherwise we  tell that   $f$ is   
$(\lambda,b)-nonsatisfactory $ and we denote by  
$NS(\lambda,b) $ the set of nonsatisfactory points.\\
$\bullet$ An edge of  $\Gamma_U$ is  
$(\lambda,b)-satisfactory $ if 
it links two 
$(\lambda,b)-satisfactory  $ configurations
otherwise it is said
$(\lambda,b)-nonsatisfactory $ .  We denote   $S^e(\lambda,b)$ 
[resp $NS^e(\lambda,b)$] the set of 
$(\lambda,b)-satisfactory $ 
[resp $(\lambda,b)-nonsatisfactory $] edges. \\
$\bullet$ A  point $u =(x,f)\in U$ is
$(\lambda,b)-satisfactory $ [resp 
$(\lambda,b)-nonsatisfactory$] 
if  $f \in S(\lambda,b)$ 
[resp $ NS(\lambda,b) $]. We denote
by $S^p(\lambda,b)$ and $NS^p(\lambda,b)$ 
the set of points which are  (or are not ) 
$(\lambda, b)-satisfactory$. \\ 
$\bullet$ A point $u=(a,f)\in U$ is said  
$b-good$ if 
 $ \underset{a}{dim} f \geq b   $
 otherwise it is   $b-bad.$
\\
\\
Let us now explain the main steps of the
proof.  We take $U \subset V(D_F)$ such that $\frac{|\partial_{D_F} U| }{|U| }
\leq \frac{1}{n}$. We begin to prove that there
exists some value of $b$ and some sequence $\lambda$ 
such that there are few points 
 $(\lambda,b)-nonsatisfactory $. Then, we
 extract a subgraph of $\Gamma_U$ where all
 points are   
 $(\frac{\lambda}{30},\frac{b}{30})-satisfactory
 $ and this allows us to obtain a lower bound of  $|U|$. We begin by the
 following lemma.\\
\\
\begin{lem}
\label{neuds} Let  $U \in V(A'\wr B'_z) $
such that $\frac{|\partial_{D_F} U| } {| U| } \leq
\frac{1}{1000n}$ then  
  \begin{enumerate}[$(i)$]

\item  $  \frac{\# \{  u=(x,f) \in U ;\  u \  is\ \lambda_x(n)- bad  \} }{\#U}
\leq \frac{1}{1000n} $
\item  $  \frac{\# \{  u=(x,f) \in U ;\  u \in
NS^p(  \lambda(n)/3,\psi(n)/3 ) \  \} }{\#U}    \leq \frac{1}{500} ,$
\end{enumerate}
where $\lambda=(\lambda_x)_x $ with $\lambda_x(n)=Fol_{B_x'}(n)$ and
$\psi(n)=Fol_{A'}(n)$.
\end{lem}

\begin{proof} $\ $ \\
For  (i) we notice that we can associate to certain bad points,
some point of the boundary of $U $.
Indeed, for  $(x,f)$  a  point,  we  call: \\ 
 $ \tilde{P}_{x,f} =\{ g(x) ; \ (x,g) \in U \ \ and \ \ \forall
 y \neq x \  g(y)=f(y)\}$ and  \\
  $ P_{x,f} =\{ (x,g); \ g(x) \in \tilde{P}_{x,f}  \}$ . Note
  that   $|\tilde{P}_{x,f}|=|P_{x,f}|.$\\
 $F_0$ stands for a set of  configurations such that: 
$$  \underset{x\in A',f\in F_0}{\dot{\bigcup } }  P_{x,f} = \{
u=(x,g) \in U  ; \  u \ is \ Fol_{B_x'}(n)-bad \}. $$ 
Take note that,  for a point $u=(x,f) $ which is $ 
Fol_{B_x'}(n)-bad, $ by the 
definition of a Folner function, we have:

$$   |\tilde{P}_{x,f}| <   Fol_{B_x'}(n).$$
So,
$$  | \partial_{B_x} \tilde{P}_{x,f} |  \geq \frac{1}{n}|   \tilde{P}_{x,f} |  $$

Now the application  $  \dot{\underset{x\in A',f\in F_0}{\bigcup}}
\partial_{B_x} \tilde{P}_{x,f}  
\longrightarrow
\partial_{D_F} U$  is injective,\\
$\hspace*{4.8cm} (g_1,g_2) \mapsto  \Bigl( (x,f_{x,g_1} ),(x,f_{x,g_2})\Bigr)$
\\
\begin{eqnarray*}
\text{where } \  (g_1,g_2)\in \partial_{B_x} \tilde{P}_{x,f} \
\text{and } \  
f_{a,h } : &v& \rightarrow f(v)  \ \text{ for}\  v \not=  a.\\
&a&\rightarrow h
\end{eqnarray*}

Hence, we have : \\
\begin{eqnarray*}
\frac{|U|}{1000n}\geq |\partial_{D_F} U| 
&\geq &
\underset{x\in A,f\in F_0}{\sum} | \partial_B \tilde{P}_{x,f}| 
\\
&\geq &\frac{1}{n}
 \underset{x\in A ,f\in F_0}{\sum} | \tilde{P}_{x,f}| \\
 &=&\frac{1}{n} \#\{ u=(a,f)\in U ;\  u \ is \ Fol_{B_a'}(n)-bad \}.
 \end{eqnarray*}
\\
For (ii), the proof splits  into three parts. \\
\begin{enumerate}[A.]
\item  Let,  
\begin{eqnarray*}Neud &=&\{  u \in U ;\  u \in
NS^p(\frac{\lambda } {3},\frac{Fol_{A'}(n)}{3} ) \  \} \\
&=&\{ u =(x,f) \in U ; f \in \ NS(\frac{ \lambda}{3},\frac{Fol_{A'}(n)}{3} )\},
\end{eqnarray*}
and  let:  
$$Neud(f) = \{ (x,f) ; (x,f)\in U  \}.$$  Notice that $p(Neud(f))=\{x; (x,f)\in U\} .$ \\
  For  $F$ a set of configurations, we call   
   $$Neud(F)= \underset{f\in F}{\cup} Neud(f). $$ 
   Note well that it is a
   disjointed union. \\
   \\
\item  Now take  $f \in NS(\frac{\lambda } {3},\frac{Fol_{A'}(n)}{3})$,  and
look at the set $p(Neud(f))$. There are only two
possibilties: \\
   -either, it gives a large part of boundary in
   'base', \\
   -either, it  gives a few part of boundary in
   'base'. If  this is the  case, taking into account
    that $f$ is not
   satisfactory,  we retrieve boundary in  'configuration'. \\
    Anyway, we get some  boundary of $U$, but  our assumptions restrict this
    contribution.\\
      \\
So we  differentiate two cases: \\

\underline{First case} : $f\in F_1 := \{  f\in
NS(\frac{\lambda}{3},\frac{Fol_{A'}(k)}{3}) ;   
\frac{\#\partial_{A'} p(Neud(f)) }{\# p(Neud(f))} > \frac{1}{n} \} .$\\

The application$ \underset{f\in F_1}{\dot{\bigcup}}\partial_{A'} p(Neud(f))
\longrightarrow  \partial_{D_F} U  \ $ is injective. \\
\hspace*{4.5cm} $       (x,y) \longmapsto \Bigl(\left(x,f\right);\left(y,f\right)  \Bigr)$ \\

So, we get:
    \begin{equation}\label{cas1main}
 |\partial_{D_F} U  | \geq  \underset{f\in F_1}{\sum} | \partial_{A'}
p(Neud( f) )| \geq  \frac{1}{n} \underset{f\in F_1}{\sum} |
p(Neud(f))|
\geq  \frac{1}{n}|Neud (F_1)|  . 
\end{equation}
\\

\underline{Second case} : $f\in F_2 := \{  f\in NS(\frac{\lambda }
{3},\frac{Fol_{A'}(n)}{3}) ;  
 \frac{\#\partial_{A'} p(Neud(f)) }{\# p(Neud(f))} \leq \frac{1}{n} \} .$\\
\\
Since $f\in NS(\frac{\lambda }
{3},\frac{Fol_{A'}(n)}{3} ) $ it follows that :
$$ \# \{ x\in p(Neud(f)) ; \ \underset{x}{dim} f \geq \frac{\lambda_x}{3}  \} < 
\frac{1}{3}Fol_{A'}(k) .$$ 
Hence,  $$ \# \{ x\in p(Neud(f)) ; \ \underset{x}{dim} f <  
\frac{\lambda_x}{3}
\} \geq |Neud(f)| - \frac{1}{3}Fol_{A'}(n) $$
(We use that  $|p(Neud(f)|=|Neud(f)|.$)\\
Since  $ f \in F_2$  and by   definition of a Folner   fonction: 
$$|Neud(f)| \geq Fol_{A'}(n).$$ 
As a result, we have:   $$ \# \{ x\in p(Neud(f) ; \ \underset{x}{dim} f  <
\frac{\lambda_x}{3}  \} \geq \frac{2}{3} |Neud(f)| .$$
   \begin{eqnarray} \label{P_f} ie :   |P_f| \geq \frac{2}{3} |Neud(f)|,
   \end{eqnarray}
with  $P_f=\{ x\in p(Neud(f) ; \ \underset{x}{dim} f  <       
\frac{\lambda_x}{3}   \}. $ \\
Let $\tilde{P}_{x,f} =\{ g(x) ; \ (x,g) \in U \ and \ \forall y \neq x \ 
g(y)=f(y)\}$. To each point of  $\partial_{B_x'} \tilde{P}_{x,f} $  we
can associate, by the same way  as before, a  point
of $\partial_{D_F} U$. So, we have:
$$|\partial_{D_F} U| \geq  \underset{x\in P_f, f\in F_2}{\sum}
|\partial_{B_x'} \tilde{P}_{x,f}|  .$$
Now for $x$ in $P_f, \  \underset{x}{dim} f  = |   \tilde{P}_{x,f} | <
\lambda_x = \frac{1}{3}Fol_{B_x'}(n)< Fol_{B_x'}(n).$
So $$   |\partial_{B_x'}   \tilde{P}_{x,f} | >\frac{1}{n}  |  
\tilde{P}_{x,f} |, $$
ie: $$   |\partial_{B_x'}   \tilde{P}_{x,f} | \geq 1.    $$ 
Then, 
 \begin{eqnarray*}
 \underset{x\in P_f, f\in F_2}{\sum} |\partial_B \tilde{P}_{x,f}| 
&\geq & \underset{f\in F_2}{\sum} \frac{2}{3} |Neud(f)|     \hspace*{1cm}
\text{by (\ref{P_f}}), \\
&\geq &  \frac{2}{3} |Neud(F_2)|
\end{eqnarray*}
We have thus 
\begin{eqnarray*}
|\partial_{D_F} U| 
\geq   \frac{1}{n} |Neud(F_2)|   \text{ for $n\geq 2$. }
\end{eqnarray*}
\item  Adding  (\ref{cas1main}) and this last
equation  and using the inequality   
$\frac{|\partial_{D_F} U| } {| U| } <
\frac{1}{1000n},$  we obtain  : \\
$$ \frac{|Neud|}{|U|} < \frac{1}{500}.$$
\end{enumerate}
\end{proof}

\begin{lem}
\label{soushypergraphes} Let  $ (\Gamma_U,
w)$   be the one dimensional skeleton with
weights $w$, constructed from  $K_U$. Let
$\eta=(\eta_a)_{a\in p(U)}.$\\
Assume that $E(\Gamma_U ) \neq \emptyset$
and   $\forall (a,f) \in U  \ 
\underset{a}{dim} f \geq  {\eta}_a >0$. 
If  the following condition is satisfied  :\\
$$ \frac{ \underset{e\in  NS_U^e(\eta ,  b ) }{\sum } w(e) }
{\underset{e\in E(\Gamma_U )}{\sum }
w(e)}<1/2,
$$\\
then there exists a  not empty subgraph  
$\Gamma'=\Bigl( V(\Gamma'),E(\Gamma') \Bigr)$
of $\Gamma_U $ such that all  edges  are   $S_U^e(\eta/10, b/10 )$.
\end{lem}

\begin{proof}
In the gaph $\Bigl(  V(\Gamma_U), E(\Gamma_U) \Bigr)$, 
we remove all points    $NS^p_U(\eta/10,b/10)$
and the adjacent edges. After this step, it may appear new points which
are $NS^P_{U_1}(\eta/10,b/10)$, 
where $U_1=U- NS^p_U(\eta/10,b/10)  $. \\
We remove once again these points and adjacent edges
and we  reiterate this
 process.\\
Let $U_i$ be  the set of points still present at step  $i$. 
$$\left \lbrace
\begin{array}{l}
U_0=U,\\
\mathrm{for} \ i\geq 1 \ \ U_{i+1}=U_i -NS_{U_i}^p(\eta/10,b/10).
\end{array}
\right.$$
It is sufficient to prove that this process stops before the graph
becomes empty.\\
Let $C_1= \underset{e\in NS_U(\eta,b)}{\sum }w(e) $ ,
$\ \ \ C_2=   \underset{ 
\underset{ \text{ at the end of the process} }
{e\in S_U^e(\eta,b) ; e \text{ removed  } }  
}{\sum } w(e) ,$ \\
 et 
  $$ C_0= \underset{
\underset{\text{ at the end of the process}  }{e\in E(\Gamma_U) ) ; e
\text{ removed } }
} 
 {\sum }
 w(e)    . $$
If we show that  $ C_2 \leq C_1$,  the propostion is proved, since : $$C_0\leq C_1 +C_2 \leq 2 C_1<
  \underset{e\in E(\Gamma_U ) }{\sum} w(e).$$ \\
Indeed, this means that it remains point(s)  not removed. ie: $\exists k_0\in \N  $ 
such that all vertices of the graph we get at step  $k_0$, are 
$S^p_{U_{k_0}}(\eta/10,b/10) $, donc 
 $ S^p_U(\eta/10,b/10).$    \\
 \\
In order to see this, let us introduce an orientation on edges removed: if $L$ 
and $Q$ are  points of the graph,  we orient the edge from  $L$  to  $Q$
if  $L$  is removed before $Q$, and we choose an arbitrary orientation if
they are removed together.   We denote   by 
$\underset{\downarrow}{L}$ the set of edges
leaving the  point $L$  and 
$\underset{\uparrow}{L}$  the set of edges ending at
point $L$, both at step $0$.

\begin{souslem}
\label{soustrucchiantsmain} 
Let $k \in \N$  and  let   $L $ stands for   a point  of the   graph 
$\Gamma_U$ (satisfying   assumptions of lemma \ref{soushypergraphes}), 
removed after  $k+1$ steps.
Suppose that   $L$  is  initially   $ S^p_{U}(\eta,b)$,  then 
$$ \underset{e \in \underset{\downarrow}{L} }{\sum} w(e)   \leq \frac{1}{2} \underset{e \in  
\underset{\uparrow}{L}}{\sum} w(e)   .$$
\end{souslem}
\includegraphics*[width=11cm]{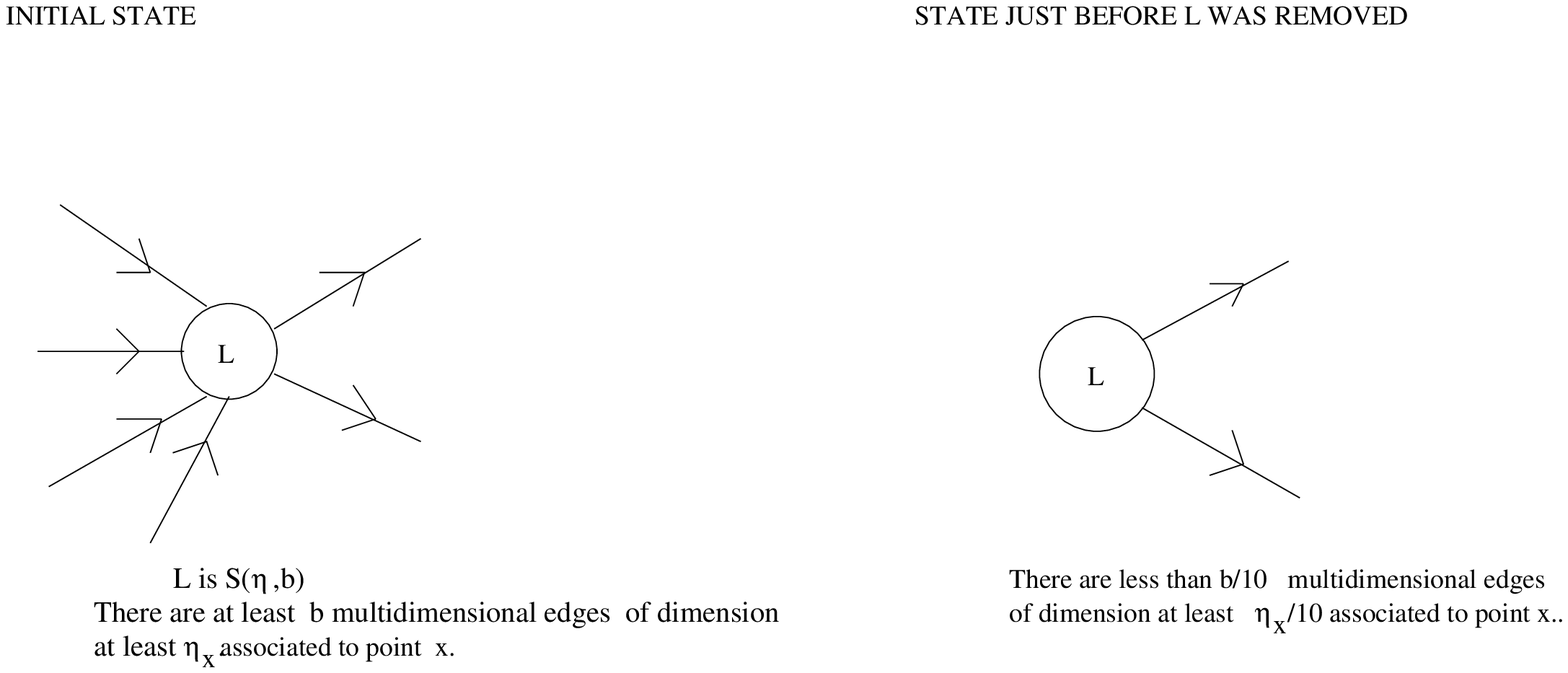}
\\
These pictures are  represented at step  $0$  on the 
left side  and at step  $k$ on the
right side.
\begin{proof}
It would be useful to notice that  for a multidimensional edge  $e$, the
sum of the weights (in the skeleton)  of edges coming from  $e$ and
adjacent to a point, is always equal to $ 1$. This is implied by our choice of
the weight. \\

The proof is divideds into  five parts.
\begin{enumerate} [A.]
\item
Let  $N_0$ the number of multidimensional edges at step $0$.  Since   $L$
is  $S^p_U(\eta,b),\ $  there are at least   $b$   multidimensional edges
attached to  $L$. So,
\begin{eqnarray} \label{N_0} N_0 \geq b .
\end{eqnarray}
Note that:
$$  \underset{ \underset {e \ contains  \ L}{ e\in E(\Gamma_U ) } }{\sum} w(e)  =N_0 .$$

\item
Let :\\
$ \underset{\downarrow}{L_1}  = \{ e\in      \underset{\downarrow}{L}, \
e \text{ coming from a multidimensional edge of} \ K_{U_k} ,\\
\hspace*{4.7cm}  \text{ associated to a point } x,  \text{ of}  \ dim\geq \eta_x/10   \}, $  \\
and \\
$ \underset{\downarrow}{L_2}  = \{ e\in      \underset{\downarrow}{L}, \
e \text{ coming from a multidimensional edge of } \ K_{U_k} ,\\
\hspace*{4.7cm} \text{ associated to a point } x, \text{ of} \ dim <
\eta_x/10   \}. $ \\
\\
We have:  $$\underset{\downarrow}{L} = \underset{\downarrow}{L_1} \cup 
\underset{\downarrow}{L_2} ,$$  because edges of  $  
\underset{\downarrow}{L}$,   are edges leaving  $L$  at step $k$.
\\
\\
\item  Since  $L$  becomes   $NS^p_{U_k}(\eta/10, b/10) $, there are less
than  $b/10 $   multidimensional edges associated to
each point $x$, of 
 dimension  at least   $\eta_x/10$. Call them  $f_1, ...,f_{q}, $  with  
$q < b/10.$    
\begin{eqnarray} \label{poid1}\underset{e\in  \underset{\downarrow}{L_1} }{\sum} w(e)  = 
 \underset{k=1..q}{\sum} \ \ 
    \underbrace {\underset{\underset{ coming \ from \   f_{k}}{e} }{\sum} w(e) }_{\leq 1} \leq q  .
    \end{eqnarray}
    (Initially this last sum was equal to  $1$,  but after removing
    some edges, this sum value becomes  less than  $1$.)\\
    \\
 \item  Let  $g_1,...,g_{h}$ be  the other  multidimensional edges attached
 to $L$ at step $k$  associated to a point $x$, and with  dimension
 strictly less  than   $\eta_x/10$.  We have $h \leq N_0-q.$\\
   Consider an edge $e$ coming from a multidimensional edge associated
to a point $x$. For all  $k=1...h$ we have:
 \begin{eqnarray} \label{poid2} \underset{\underset{coming \ from \  g_{k}}{e} }{\sum} w(e)  \leq 
 \frac{1}{\eta_x} \frac{\eta_x}{10} \leq \frac{1}{10} . 
 \end{eqnarray}  
Indeed, firstly since all configurations (relatively to this edge $e$) have initially dimension at least $\eta_x$
we deduce that   $ w(e) \leq 1/\eta_x$. And secondly  a multidimensional
edge of dimension less than $\eta_x/10$ gives less than $\eta_x/10$ 
edges in the skeleton.\\

\item Finaly by  (\ref{poid1}) and  (\ref{poid2}),  we get:
\begin{eqnarray*} 
\underset{e\in  \underset{\downarrow}{L} }{\sum} w(e)  &= &
\underset{e\in  \underset{\downarrow}{L_1} }{\sum} w(e)  +
\underset{e\in  \underset{\downarrow}{L_2} }{\sum} w(e) \\
&\leq&  q + (N_0-  q   )\frac{ 1}{10} \\
&=&  \frac{1}{10} N_0+ \frac{9}{10} q\\
&=& \frac{19}{100} N_0.
\end{eqnarray*}
($q < b/10 \leq N_0/10$ by (\ref{N_0}).)\\
\\
So,
$$ \underset{e\in  \underset{\downarrow}{A} }{\sum} w(e)  \leq
\frac{19}{100} N_0  \ \ and \ \
 \underset{e\in \underset{\uparrow}{A}}{\sum} w(e)  \geq  N_0 - \frac{19}{100}N_0=\frac{81}{100}N_0    .$$
So,  
\begin{eqnarray*}
\underset{e \in \underset{\downarrow}{A} }{\sum} w(e) 
  \leq \frac{19}{81}\underset{e \in  
\underset{\uparrow}{A}}{\sum} w(e) 
\leq \frac{1}{2}  \underset{e \in  
\underset{\uparrow}{A}}{\sum} w(e) .
\end{eqnarray*}

\end{enumerate}
\end{proof}

To finish the proof, let us consider:\\
$D_1=\{ \text{ vertices  \ removed \ at \ step} \  1 \},$ and for   $i\geq 2$ \\
$D_i=\{ \text{vertices } \ S^p_{U}(\eta,b) \ \text{removed \ at \ step }\  i \},$ \\
$F_i= \{ \text{edges \ between }\ D_i \ and \ D_{i-1}  \},$ \\
$ F_i'= \{ \text{edges \ leaving  }\ D_{i-1}  \}     $.\\
\\
Note that   $F_i \subset F_i'$  and that  the edges of    $F_i'$ are removed.\\


\includegraphics*[width=11cm]{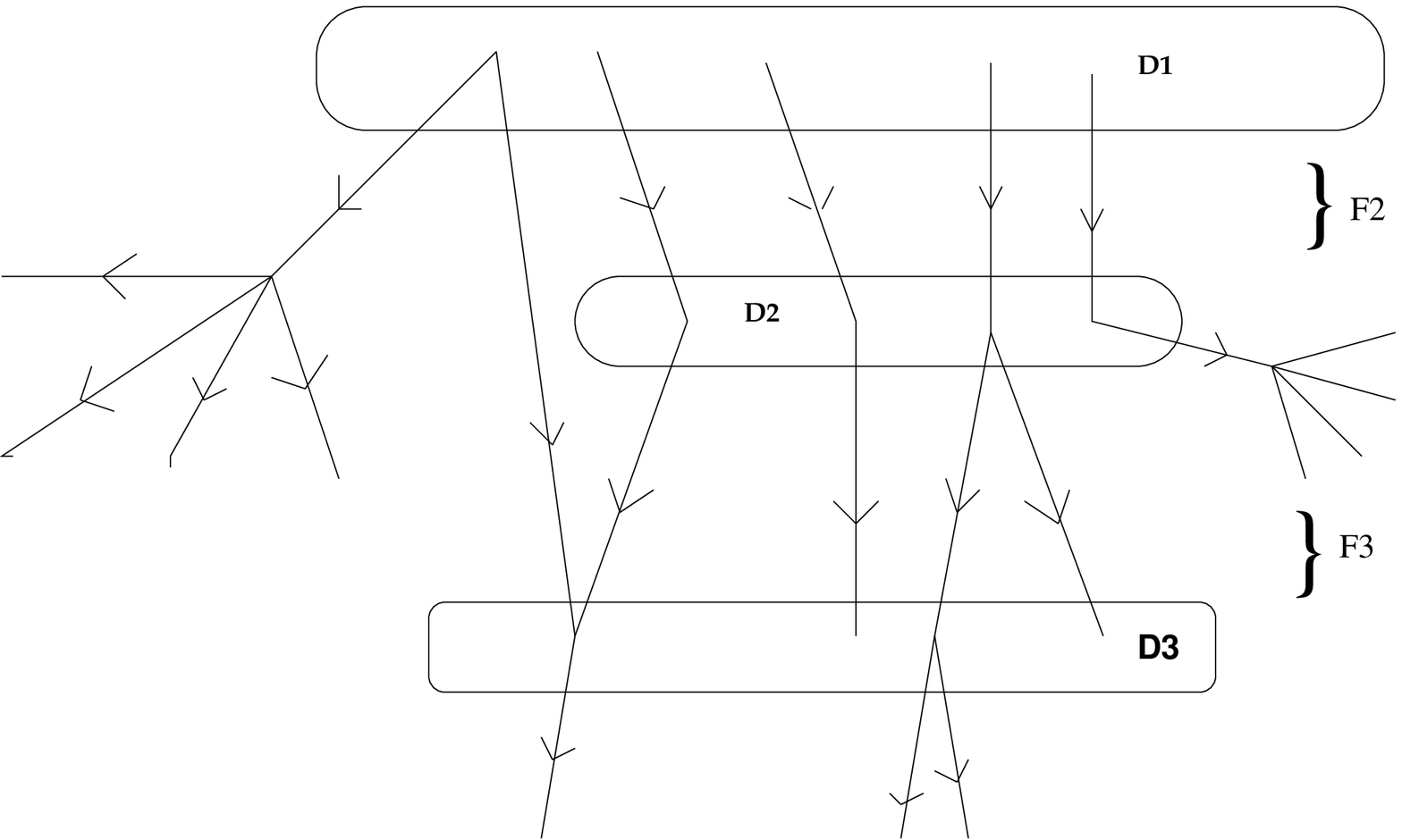}
\\
The proof ends up in four parts:
\begin{enumerate} [A.]
\item
Apply  sublemma  \ref{soustrucchiantsmain}  to  each point of  $D_i$, in
the graph staying at step   $i-2$.  
(Each point of   $D_i$  is  $S(\eta,b)$.) We get :
$$\forall i \geq 2 \ \  \underset{e\in F_{i+1}'}{\sum} w(e)  \leq \frac{1}{2} \underset{e\in F_i}{\sum}w(e) . $$

So, $$ \ \  \underset{e\in F_{i+1}'}{\sum} w(e) \leq (\frac{1}{2})^{i-1}  \underset{e\in F_2}{\sum} w(e) . $$
(We use that   $F_i \subset F_i'.$)\\
Hence,
\begin{eqnarray*}  \ \underset {e\in   \underset{i\geq 3}{\cup} F_i'}{\sum} w(e)   &\leq&
(\underset{i\geq 1} {\sum} (\frac{1}{2})^i)  \underset{e\in  F_2}{\sum} w(e) \\
&=&  \underset{e\in  F_2}{\sum} w(e) .
\end{eqnarray*}
\item Now, an edge of   $F_2$ is $NS_U^e(\eta,b)$
since  if it was
$S_U^e(\eta,b)$, it would link two points  
$S_U^p(\eta,b)$  and in particular    points  of   $D_1$  would have been  
$S_U^p(\eta,b),$  then  $S_U^p(\eta/10,b/10)$ 
and so would  not have been removed.
In consequence :
 $$\underset{e\in F_2}{\sum} w(e)  \leq \underset{e\in NS^e(\eta,b)}{\sum} w(e)=C_1  .$$
\item Besides, all removed edges   $S_U^e(\eta,b)$
are in some  $F_i'$
with   $i\geq 3$,  so
$$C_2=\underset{ 
\underset{e \in S_U^e(\eta,b) }
{ e \text{ \ removed \ at  \  the   \  end  \ of  \  the \    process} }}
{\sum}
w(e)      
\leq  \underset{e\in \underset{i\geq 3}{\cup} F_i'}{\sum} w(e) .$$ \\
\item  Hence,  $C_2\leq C_1$,  which achieves the proof. 
\end{enumerate}
\end{proof}

Now, we use the  following lemma to get a lower bound of the volume of  $U$.
\begin{lem} \label{choixsuccessifsuite} Let  $N:\R_+ \longrightarrow \R_+,$ 
a non decreasing function.
\\ Let  us take $b\in \N^*$  and   $\A$  a not empty set of configurations such
that :\\
 $ \forall f \in \A \ \exists x_1,x_2,...,x_b \in \Z \ \text{such that  }
  \forall i \in[|1;b|]   \ g_i \in \A $ 
\\
where  $g_i$ is one of the following functions, defined from $f$ by : \\
$g_i (x)=
\begin{cases}
 \;  f(x) & \text{if $x \neq x_i ,$} \\
  \; \text {there are   $N(|x_i|)$ possibilities for   $g_i(x_i)$  }  &
  \text{if
   $x=x_i,$  }
\end{cases}$
\\

then $ |\A| \geq  
 \left\lbrace
\begin{array}{l}
N(0) \ \Bigl( N(1) N(2)...N(\frac{b-1}{2}) \Bigr)^2   \hspace{1.9cm} \ \ 
\text{if \ b \ is\ odd,} \\
 N(0) \ \Bigl( N(1)N(2)...N(\frac{b-2}{2})\Bigr)^2 N(\frac{b}{2}) 
 \hspace{1.1cm} \ \   \text{if \ b \ is  \ even.}\\
\end{array}
\right.$
\end{lem}
\begin{proof}
We  will proceed  by induction on   $b$.\\
If $b=1$  it is true,  since   $N$ is non decreasing on $\R_+$ .\\
Assume   $b\geq 1$  and consider a   point $x_0 $  in the base such that:\\
 $\bullet |x_0|\geq \frac{b-1}{2} $  if  $b$ is odd and   $|x_0| \geq
 \frac{b}{2}$ if  $b$ is even. \\
  $\bullet $  And there exists   $f_1,...,f_{N(|x_0|)} \in \A $ satisfying  
    $\forall i \in [|1;N(|x_0|) |] \ \  f_i(x_0) $ 
 range among the   $N(|x_0|) $  possible images.\\
 For   $i\in [| 1;N(|x_0|) |]$,  we denote by 
$\A_i$ the set $ \{f\in \A ; f(x_0)=f_i(x_0)   \}, $ 
which is not empty.\\
We have   $\A= \underset{1\leq i \leq N(|x_0|) }{\dot{\bigcup}}  \A_i.$ \\
Besides,  the  $\A_i$  satisfies  the induction assumption with constant  $b-1$. \\
So, if for example   $b$ is odd, $N(|x_0|) \geq  N(\frac{b-1}{2}) $  and
we have: 
\begin{eqnarray*}
|\A| &=&\underset{1\leq i \leq N(|x_0|) } {\sum} | \A_i| \\
 &\geq &  
 \underset{1\leq i \leq N(|x_0|)} {\sum}  N(0) \ 
 \Bigl(   N(1)...N(\frac{b-3}{2})  \Bigr)^2 \ N(\frac{b-1}{2}) \\
 &\geq& N(0) \ \Bigl(   N(1)...N(\frac{b-3}{2})  \Bigr)^2 \ 
 N(\frac{b-1}{2})  N(x_0)  \\
 &\geq&   N(0)   \Bigl(   N(1)...N(\frac{b-1}{2})  \Bigr)^2.     
 \end{eqnarray*}

The proof unfolds the same way when $b$ is an even
number. 
\end{proof}

\subsubsection{Proof  of   (i) of the  proposition
\ref{existence} :} $\ $ \\
$\bullet $ \textbf{Lower bound of Folner function.}\\ For the lower bound of $Fol_{D_F}$, take $U \subset V(A'\wr B'_z)$  such that  
$\frac{|\partial_{D_F} U|}{|U|} \leq 
\frac{1}{1000n}$
Let $\tilde{K}=\Bigl(V(\tilde{K}), \xi(\tilde{K} ) \Bigr)$ the
subhypergraph of $K_U$  constructed with  points $(x,f)$ which are   $Fol_{B_x'}(n)/3-good$.
 $\tilde{K}$  is not empty, since by the part  (i) of the  lemma \ref{neuds}
  $|V(\tilde{K})|  \geq
(1-\frac{1}{1000n})|U|.$ \\
 Then we have:
\begin{eqnarray*}
 \underset{ e\in E(\Gamma(\tilde{K}) )\cap NS^e(
 \frac{\lambda(n)}{3},\frac{\psi(n)}{3} ) }{\sum} w(e)
  &\leq& \frac{1}{2} \#\{ u\in U ; NS^p \Bigl(  
  \frac{\lambda(n)}{3},\frac{\psi(n)}{3}           \Bigr)  \} \\
    &\ &\hspace{4.4cm}\mathrm{ by \ remark  \ (\ref{remcardinalpoids} )}\\
&\leq &  \frac{1}{1000}|U  | \hspace{3.1cm} \mathrm{ by \ lemma  \ \ref{neuds}} (ii)\\
  &\leq&  \frac{1}{ 1000-\frac{1}{n} } \ \#\{ u=(x,f)\in U, \
  \frac{\lambda_x(n)}{3}-good\} \\
  &\ & \hspace{4.35cm} \mathrm{ by \ lemma \  \ref{neuds}} (i)\\
 &= &\  \frac{2}{ 1000-\frac{1}{k} }   \underset{e\in  E(\Gamma(\tilde{K} )  ) }{\sum} w(e)  \\
 &\leq &\  \theta  \underset{e\in  E(\Gamma(\tilde{K} )  ) }{\sum} w(e) .    \\
\end{eqnarray*}
 with  $\theta= \frac{2}{ 999 }<  \frac{1}{2}$, so  lemma \ref{soushypergraphes}  can be applied to  $\tilde{K}$, to deduce there
 exists a subgraph   $K'=(V(K'),E(K'))$ of   $ \tilde{K} $ such
 that all edges are   $S^e(\lambda(n)/30, \psi(n)/30)$.\\
 Then by lemma \ref{choixsuccessifsuite} applied with 
  $N(|x|) = Fol_{B_x}(n)/30$ to  the set of configurations
  relatively to $K'$, we deduce  for large enough $n$ :
  $$|U| \geq l(0) \ \Bigl( l(1)...l(\frac{\psi(n)}{40}) \Bigr) ^2 =
 \frac{F(1)}{F(0)} \ \Bigl(\frac{F(2)}{F(1)}  ...\frac{F(n/40 +1
 )}  {F(n/40 )  } \Bigr)^2.$$
 (We use that for $k \geq 3,\ Fol_{B_x'}(k)= |B_x'|=l(|x|)  =
 \frac{F(|x| +1)}{F(|x| )}$.)\\
So, $$|U| \geq c F(n/40)^2  \succeq   F(k).$$ 
(Since   $F(x)=e^{cx^
{    \frac{2\alpha }{1-\alpha}    }
}$ we have   
$ F\approx F^2.$) \\
ie :$$ Fol_{D_F}(k) \succeq F(k).$$
$\bullet $\textbf{Upper bound of Folner function.} \\For the upper bound of the Folner  fonction of  $D_F$,
we take:
 $$U= \{(a,f) ; 0\leq a \leq n \ ; supp(f) \subset 
[|0;n|] \} .$$
On a  $$|U|= n F(n)  \ \text{et} \ 
|\partial_{D_F} U| /|U| \leq c/n,$$
so, $$Fol_{D_F} (n) \leq n F(n) \preceq F(n) .$$
$\bullet $ So the graph $D_F $ has the expected Folner function
on the case $\alpha >1/3.$\\
\\
\subsubsection{Proof  of   (ii) of the  proposition \ref{existence} :} 
We proceed in 5 steps.
\begin{enumerate}[A.]
\item  Let  $d_0=(0,f_0)$ where  $f_0$  is the null configuration.\\
Let $H_n=(K_n,g_n)$  the random walk on   $ D_F$
starting from  $d_0$ which jumps uniformly on the set of points
formed by the point where the walker is and its neighbors. \\
This random walk admits a reversible measure  $\mu$ defined by 
$\mu(x) =\nu_{D_F}(x) + 1$. Note that for all  $x\in V(D_F), 
\mu(x) \leq 5 .$\\
\item  Using reversiblity, we can write, 
\begin{eqnarray*}
p^{D_F}_{2n}(d_0,d_0) &=& \underset{z}{\sum} p_n^{D_F}(d_0,z)
p_n^{D_F}(z,d_0)\\
&\geq& \underset{z \in A}{\sum}  p^{D_F}_n(d_0,z) ^2  \ \frac{\mu(d_0)}
{\mu(z)}\\
&\geq& \frac{\mu(d_0 )}{\mu(A)}[ \underset{z \in A}{\sum} p^{D_F}_n(d_0,z)]^2\\
&\geq& \frac{\mu(d_0)}{\mu(A)}[  \P^{D_F}_{d_0} (H_n \in A)]^2,
\end{eqnarray*}
where  $A$ is  some subset of  $V(D_F).$\\ 
Choose  $A=A_r=\{  (a,f) \ ; \ | a| \leq r \ and \ supp(f)\subset [-r,r] \}. $ \\
\item The structure of edges on $D_F$ implies:  
\begin{eqnarray*}
\P^{D_F}_{d_0}(H_n \in A_r) &\geq& \P^{D_F}_{d_0} 
( \forall i \in [|0,n|] \ |K_i|\leq r  ) \\
&\geq&  \P_0^K ( \forall i \in [|0,n|] \ |K_i|\leq r  ),
\end{eqnarray*}
where  $\P_0^K$ is the law of $(K_i)$ which is again a random
walk with  probability transitions that can be represented    for
$n$ large enough by :\\
$$\includegraphics*[width=6cm]{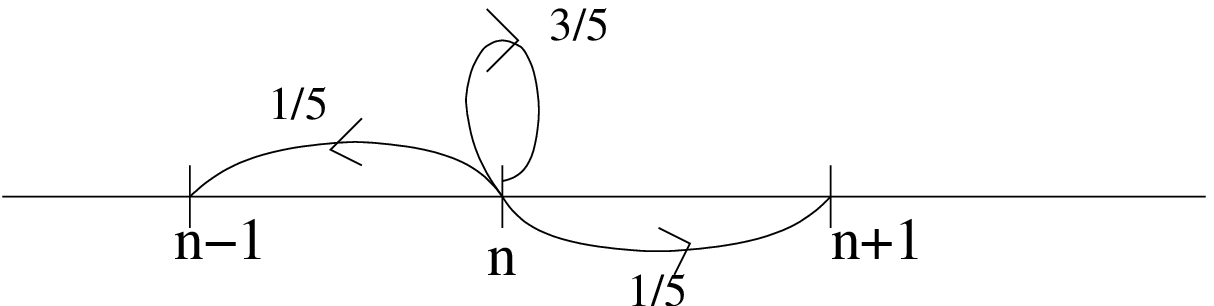}$$
Indeed, as soon as  $l(|n|) > 3$,  the  point $(n,f)$ has  2 
neighbors in "configuration", 2 neighbors in "base"  and itself
as neighbor.
For this walk we can prove (as in proposition 5.2  in  \cite{KL}) that :
$$\exists c>0, \forall n\geq 0 \  \ \     
\P_0^K ( \forall i \in [|0,n|] \ |K_i|\leq r  ) \geq   e^{-c(n/r^2 +r)}.$$
In fact, a better bound holds $\P_0^K ( \forall i \in [|0,n|] \
|K_i|\leq r  ) \geq   e^{-cn/r^2}$ (see lemma 7.4.3 of 
\cite{SPbook})  but it is not necessary here.\\
Thus,
\begin{eqnarray} \label{Z2}
\P^{D_F}_{d_0}(H_n \in A_r) \geq e^{-c(n/r^2+r)}.
\end{eqnarray}
\item  Compute now $\mu(A_r)$,  we have:
 \begin{eqnarray*}
 \mu(Ar) &\leq& |Ar| \max_{A_r} \mu  \\
 &\leq& (2r+1) \frac{F(1)}{F(0)} (\underset{k=1..r}{\prod} \frac{ F(k+1}{F(k) })^2 \times 5 \\
  &\leq&  C r F(r+1)^2 \\ 
  &\preceq&  F(r).  
 \end{eqnarray*}
( This last inequality comes from the form of $F(r)$ in
 $e^{cr^ {\frac{2\alpha}{1-\alpha } } }$.)
\item Gathering the results, by  inequality  (\ref{Z2}) and the
fact that  $\frac{2\alpha}{1-\alpha }\geq
1$,   we deduce that it  exists $c>0$ such that:
$$   p^{D_F}_{2n}(d_0,d_0)  \geq 
 e^{ -c( \frac{n}{r^2} +  r^ {\frac{2\alpha}{1-\alpha } }   )}  .$$
The  function  $r \mapsto
\frac{n}{r^2}+r^{\frac{2\alpha}{1-\alpha} }$ is  minimal  for  
$r$  like 
$n^{\frac{1-\alpha} {2}}$.\\
So , it exists $c>0$ such that:
$$p^{D_F}_{2n}(d_0,d_0) \geq e^{-c n^{\alpha}} .$$
\end{enumerate}

\begin{rem}  Note that by proposition \ref{propfol} and with
our estimate of $Fol_{D_F},$  we have for all $x,y$ in $D_F,$
$p^{D_F}_{2n}(x,y) \preceq e^{- n^{\alpha}}$. So
$p^{D_F}_{2n}(d_0,d_0) \approx e^{- n^{\alpha}}$
\end{rem}
\subsection{case $ 0 \leq \alpha \leq \frac{1}{3} $ \label{d}}
\subsubsection{ Construction of the graph and preliminary lemmas.}
Consider the general following context: let  $A$ and  $B$
 two  graphs and  $\phi $ an  application  $A\rightarrow A'$.
 Now we look at the graph such that:\\
- the  points  are  elements of  $(A\times B^{A'})$, \\
- edges are  couple $((a,f);(b,g))$ such that :\\
(i) either   $\forall x \in A'$, $f(x)=g(x)$ and  $a$  is
neighbor of  $b$ in $A$.\\
(ii) either  $a=b$ and    $\forall x \neq \phi(a)\ 
f(x)=g(x)$ and 
$f(\phi(a)) $ is neighbor of  $g(\phi(a))$ in  $B$.\\
\\
Such graphs are called generalized wreath products.\\
If $A'=A$  and  $\phi=id$  we retrieve our ordinary  wreath
products.\\

Case which interest us is when   $A=A' =(\Z,E(\Z) )$ and 
$B$ is the Cayley  graph of  $\frac{\Z}{2\Z}$ with 
$\bar{1}$ as generator. \\
 To define  $\phi: \Z\rightarrow \Z$, it is sufficient to
 give the following sets $A_i=\{ x; \phi(x)=i \}$, which
 should form  a  partition  of $\phi(\Z)$ (which is here
 $\Z$). Let  $\A=\{A_i\}$, we note  $ A {\wr}_{\A} B $  the generalized wreath
product considered.\\
\\
Let  $\beta= \frac{2\alpha}{1-\alpha} <1.$\\
If we want a Folner function like $e^{n^{\beta} }$, 
we should construct  $\phi$  (or the partition $\A$) 
with some redundancies. Suppose for example that Folner
sets are : 
\begin{eqnarray} \label{U_n} U_n=\{ (a,f);\ a\in[-n;n] \ et \ supp(f) \in [-n;n] \},
\end{eqnarray} we should have 
 $$\#\phi ([|-n;n|] ) = \{ i;\ A_i\cap[-n;n] \neq \emptyset\}\approx n^{\beta}.$$ 
\\
For  $\Omega\subset A$, it would be useful to introduce: 
$$N^{\A} (\Omega)=\#\{  i;\ A_i\cap \Omega \neq \emptyset \},$$
and $$S_j( \Omega )=\#(A_j\cap  \Omega ) .   $$
In particular, let:
$$N^{\A} (k,k+m)= N^{\A} ( [k,k+m] ) \ \  \text{et} \ \  S_j(k,k+m)= S_j( [k,k+m[) .   $$
The following  lemma   gives us the construction of the
partition which answers to our problem. 
\begin{lem}  \label{partition}$\ $ \\
Let $g :\N\rightarrow \N$ increasing with $g(1)=1$ such
that  for all $n$  in  $\N$, $$g(2n)\leq 2g(n).$$  
Then there exists a partition $\A_g =\{A_i\}$ of  $\Z$
satisfying: \\
(i)  for all  $m\geq 0 $ and for all $k$ in $\Z$,
$$ N^{\A_g} (k,k+m) \approx g(m),$$
(ii) there  exists $K>0$ such that for all  $ m\geq 0$, for
all  $k$ in $\Z$ 
and for all  $i,j $ in  $S_j(k,k+m)\neq 0$:
$$ \frac{S_i(k,k+m) }{S_j(k,k+m)}\leq K.$$
\end{lem}

\begin{proof} $\ $ \\
\begin{enumerate}[A.]
 \item We first define  partition on  intervals
 $[1,2^s]$  ($s\geq 0$) by induction on  $s$, such that :
\begin{eqnarray*}
(\mathcal{P}_s) \ \ \ \ \begin{cases} \;N^{\A_g} (1,2^s)= g(2^s),\\
\;\frac{S_i(1,2^s) }{S_j(1,2^s)}\leq 2  \ \ \  \text{for } S_j(1,2^s)\neq 0 .
\end{cases}
\end{eqnarray*}
$\bullet$ For $s=0$,  we put the point  $1$  in some 
$A_i$, since $g(1)=1$ 
(for example $A_1$).\\
$\bullet$ Let  $s\geq 1$ and suppose now the partition is
built on $[1,2^s]$. 
We  extend this  partition to $]2^s,2^{s+1}]$.\\
Let $A_1, A_2,...,A_{g(2^s)} $ the partition  on 
$[1,2^s]$ given by induction assumption. \\
Rank by decreasing cardinal these sets: 
$A_{i_1}, A_{i_2},...,A_{i_{g(2^s)} }. \ \ \ $  (*) \\
ie:  $ \ \ \#(A_{i_1}\cap [1,2^s]) \geq  \#(A_{i_2}\cap [1,2^s])\geq...\geq
\#(A_{i_{g(2^s)} } \cap [1,2^s]) $. \\
(*)  is only  to  get  (ii). \\ \\
Let $j\in ]2^s,2^{s+1}]$, there exists  $i_k$  such that  $j-2^s \in  A_{i_k},$\\
-if $k > g(2^{s+1} ) -g(2^s) $, we put  $j$  in  $A_{i_k}$,\\
-otherwise,  we put  $j$  in a  "new " class, $j\in A_{g(2^s) + k}$. \\
\\
Thus we have :
\begin{eqnarray*}
N^{\A_g}(1,2^{s+1})& =& N^{\A_g}(1,2^s) + \#\{ k\in [1,g(2^s) ]; \ k\leq  g(2^{s+1})-g(2^s) \}\\
&=&  g(2^s)+ g(2^{s+1})-g(2^s)\\
&=& g(2^{s+2}).
\end{eqnarray*}
Besides, note that by  construction either 
$S_i(1,2^{s+1})= S_i(1,2^s)$ or either  $S_i(1,2^{s+1})= 2S_i(1,2^s)$.
So the second assertion of ($\mathcal{P}$)  is well
satisfied at the rank  $s+1$, 
except when  $S_i(1,2^{s+1})$  has doubling and 
$S_j(1,2^{s+1})$ is unchanged.
But in this case, by (*)  we have  $ \#( A_i\cap[1,2^s])\leq \#(A_j\cap[1,2^s]) $,
that could be written $S_i(1,2^{s})\leq S_j(1,2^{s})$.
So,
$$ \frac{S_i(1,2^{s+1}) }{S_j(1,2^{s+1})}= 2\frac{S_i(1,2^s) }{S_j(1,2^s)}
\leq 2.$$

\item We end up the construction of the  partition on
$\Z$ as follow:  for $j\leq 0, \ $ we put $j\in A_i$
where  $-j+1\in A_i$.  we call  $\A^g$ this partition.\\
\\
\item Let us check conditions  (i) and  (ii).\\
First, notice that for all integers  $A$ and for all
$s\geq 0$,   partitions on$[1,2^s]$   and  $[A2^s
+1,(A+1)2^{s+1} ]$ are equivalents.
And in particular we have:
 \begin{eqnarray}
 \label{re1}  N^{\A_g}(0,2^s)= N^{\A_g}(2^s A, 2^s(A+1) ),\\
 \label{re2}
\text{  et} \ \   \frac{S_i( 2^s A, 2^s(A+1)) }{S_j(  2^s A, 2^s(A+1))}
\leq 2 . 
\end{eqnarray}
Consider $k\in \Z$  and  $m\geq 0$.\\
Let $s\geq 0$ be  such that $2^{s-2}< m\leq 2^{s-1}$ and let  
$A=\min\{D;\ k\leq D2^{s-2} \}$.
 We have $[A2^{s-2}, (A+1)2^{s-2}]\subset [k,k+m]$  and
 then
\begin{eqnarray*} 
N^{\A_g}(k,k+m ) &\geq& N^{\A_g}(2^{s-2} A, 2^{s-2}(A+1) ) \\
&=&  N^{\A_g}(0,2^{s-2}) \\
&=& g( 2^s/4) \\
&\geq& g(m/4) \\
&\succeq& g(m).
\end{eqnarray*}

Let $B=\max\{D;\ D2^{s-1 } \leq k \}$, we have  $[k,k+m]\subset 
[B2^{s-1}, (B+2)2^{ s-1} ]$. So,
\begin{eqnarray*} 
N^{\A_g}(k,k+m ) &\leq& N^{\A_g}(B2^{s-1} , (B+2)2^{s-1}  ) \\
&=&  N^{\A_g}( B2^{s-1} , (B+1)2^{s-1} )  + 
N^{\A_g}( (B+1)2^{s-1} , (B+2)2^{s-1} )\\
&=& 2 g( 2^{s-1} ) \\
&\leq&  2 g(2m)\\
&\preceq & g(m).
\end{eqnarray*}
That proves  (i). \\
 \\
Let now  $C=\max\{D;\ D2^{s-3}\leq k \}$,  by the definition
of $s$,  it is easy to verify that :
\begin{eqnarray} \label{incl}[(C+1)2^{s-3}, (C+2)2^{s-3}  ] \subset[ k,k+m]
 \subset [ C2^{s-3}, (C+5)2^{s-3} ] .
 \end{eqnarray} 
$$\includegraphics*[width=11cm]{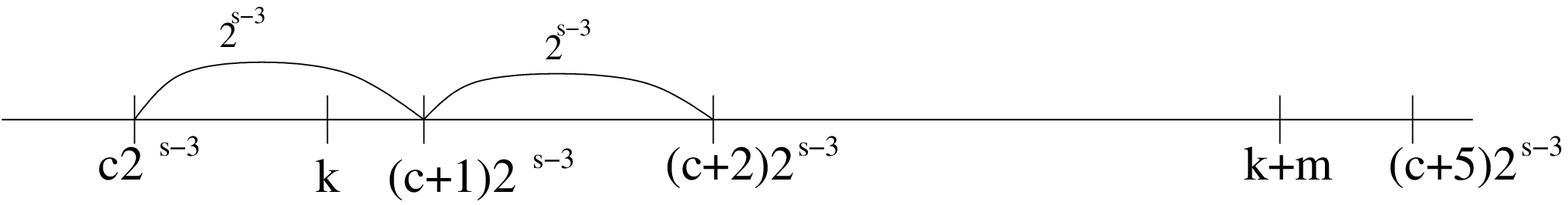}$$
 Let  $i,j$  be the subscript  which index the  partition
 such that 
  $S_i(k,k+m)\neq 0$ and  $S_j(k,k+m)\neq 0$,   we can
  write,
 \begin{eqnarray}
  \nonumber S_i(k,k+m) &\leq&  S_i(C2^{s-3}, (C+5)2^{s-3}    ) \\
 \nonumber &\leq&  2  S_j(C2^{s-3}, (C+5)2^{s-3}    )   \ \ \ \text{par (\ref{re2})} \\
  \label{ouarf}&=& 2 [ S_j( C2^{s-3}, k) + S_j(k,k+m) + S_j(k+m, (C+5)2^{s-3} )]. 
 \end{eqnarray}
 Consider the  terms $S_j( C2^{s-3}, k)$ and $ S_j(k+m, (C+5)2^{s-3} )$.\\
 First we have $S_j( C2^{s-3}, k)\leq S_j(C2^{s-3},(C+1)2^{s-3} )$.\\
 Besides, there exists  $j_1$  such that  
$$ S_j(C2^{s-3},(C+1)2^{s-3} )= S_{j_1} ((C+1)2^{s-3},(C+2)2^{s-3}   ).$$ 
We deduce  
\begin{eqnarray*}
 S_j( C2^{s-3}, (C+1)2^{s-3} )  &=&  S_{j_1} ((C+1)2^{s-3},(C+2)2^{s-3}   ) \\
 &\leq& 2 S_{j} ((C+1)2^{s-3},(C+2)2^{s-3}   ) \ \
 \text{by (\ref{re2})}\\
 &\leq&  2S_j(k,k+m) \ \ \ \text{by the first inclusion  of  (\ref{incl})}
 \end{eqnarray*}
 By using the same approach, we prove,
 $ S_j(k+m, (C+5)2^{s-3} ) \leq 2 S_j(k,k+m) .$
 Finaly with (\ref{ouarf})  we get,
 $$S_i(k,k+m) \leq K S_j(k,k+m) \ \ \ \text{with}\  K=10.$$
 That proves  (ii). 
 \end{enumerate}
\end{proof}
\begin{rem}  The property  (ii) of lemma \ref{partition},
can be extend immediatly for all finite set  $\Omega$.
Indeed, we have for each connected component  $\Omega^s$
of  $\Omega$, $S_i(\Omega^s) \leq K S_j(\Omega^s)$.
Then summing on  $s$, we get  $S_i(\Omega) \leq K S_j(\Omega)$
\end{rem} \label{theremark}
 Before showing that the graph 
 $A\wr_{\A_g} B$  is solution of our problem,  let us
 notice the following property of the  partition $\A_g$, 
 that will be useful in the next. 
 \begin{lem} \label{noel} Let $g$  satisfying assumptions
 of property 
 \ref{partition} and  $\A_g=\{A_i\}$  the  associated partition.
 There  exists  constants $c_1,c_2>0$ such that for all  $\Omega\subset\Z$, 
  satisfying  $\frac{|\partial_{A'} \Omega|}{|\Omega|}\leq \frac{1}{k},\ $ 
 for all 
  $\Omega_{\delta} \subset \Omega$  such that     
  $|\Omega_{\delta}|\geq \delta|\Omega|$,  ($\delta >0$) 
  we have:    
 $$\#\{i;\ A_i\cap \Omega_{\delta}\neq \emptyset \}
 \geq c_1\frac{\delta}{2K} g(c_2 Fol_A(k)  ) ,$$
 where  $K$  is the  constant which appears in the item 
 (ii) of  lemma \ref{partition}.
 \end{lem}

  \begin{proof} $\ $ \\
 \begin{enumerate}
 \item  Let  $\Omega \subset \Z$ such that  
 $\frac{|\partial_{A'} \Omega|}{|\Omega|}\leq \frac{1}{k}.$  There 
 exists at least one connected component   $\Omega^{s_0}$ of  $\Omega$ 
 such that  
  $\frac{|\partial_{A'} \Omega^{s_0}|}{|\Omega^{s_0}|}\leq \frac{1}{k}$ 
  and so  $|\Omega^{s_0}| \geq Fol_A(k)$. 
 \item  Take for  $c_1 $ et $c_2$  the  constants verifying
 $ N^{\A_g}(k,k+m)\geq c_1g(c_2m), \ $ for all  $k$ in $\Z$  and  $m$
 in $\N$.
  \item There  exists $i_0$ such that  $0<|A_{i_0}\cap\Omega|\leq \frac{ |\Omega|}
  {c_1g(c_2 Fol_A(k))}.$\\
 Indeed, if for all $j$  such that $|A_j\cap\Omega|>0$ we had  
  $|A_j\cap\Omega| > \frac{|\Omega|}{c_1g(c_2  Fol_A(k))}$ then we would
  have had ,
  \begin{eqnarray*}
  |\Omega|&=& \sum_j |A_j\cap \Omega |\\
  &>&  N^{\A_g}( \Omega) \frac{|\Omega|}{c_1g(c_2 Fol_A(k))} \\
  &>& N^{\A_g}( \Omega^{s_0}) \frac{|\Omega|}{c_1g(c_2 Fol_A(k))}\\
  &>& |\Omega|  \ \ \ (\text{by the choice of  }\  c_1 \ \text{et} \ c_2.)\\
  & \text{Absurd.}&
  \end{eqnarray*}
 \item We deduce that for all  $i, \ |A_i\cap\Omega| \leq 
 \frac{K |\Omega|}{ c_1g(c_2Fol_A(k)  )}.$ \\
 Indeed,  by  remark \ref{theremark}, for all $i$  we can write :
   \begin{eqnarray*}
  |A_i\cap \Omega| = S_i(\Omega)\leq K S_{i_0}(\Omega) =K |A_{i_0}\cap \Omega|
  \leq  \frac{K |\Omega|}{ c_1g(c_2Fol_A(k) )}. 
  \end{eqnarray*}
  \item Assume now that 
   $\#\{ i;\ A_i\cap \Omega_{\delta}\neq \emptyset \}
 \leq c_1\frac{\delta}{2K} g(c_2 Fol_A(k)) .$  Then we  have  successively,
 \begin{eqnarray*}
  \delta|\Omega|&\leq & |\Omega_{\delta} | \\
  &=&  \underset{i;\ A_i\cap\Omega_{\delta} \neq \emptyset}{\sum}
   |A_i\cap \Omega_{\delta} |\\
  &\leq &   \#\{ i;\ A_i\cap \Omega_{\delta}\neq \emptyset \} \times
  \max_i |A_i\cap \Omega_{\delta}| \\
  &\leq &   \#\{ i;\ A_i\cap \Omega_{\delta}\neq \emptyset \} \times
  \max_i |A_i\cap \Omega| \\
  &\leq& c_1\frac{\delta}{2K} g(c_2 Fol_A(k) ) \times  
  \frac{K|\Omega|}{c_1g(c_2 Fol_A(k))  } = \frac{\delta |\Omega|}{2}.\\
  &\text{Absurd.}&  
  \end{eqnarray*}
\end{enumerate}
\end{proof}

 Take now  $g:x\rightarrow x^{\beta}$. Since  $\beta <1$,  assumptions of
  lemma  \ref{partition}  are satisfying. Let  $D_F=A\wr_{\A_g} B$, in
  the following lines we are going to prove that this 
 graph is solution of propostion \ref{existence}.
\subsubsection{proof of (i) of proposition
\ref{existence}} \label{ipropdur}$\ $ \\
$\bullet $\textbf{Upper bound of Folner function}\\
 Using  the sets  $U_n$ defined by  (\ref{U_n}), we get upper bound of
 Folner function. .
$$ Fol_{D_F} (n) \preceq |U_n| = (2n+1) 2^{ N^{\A_g}(-n,n)} \approx 
e^{n^{\beta}}.$$
$\bullet $\textbf{Lower bound of Folner function}\\
We get the  lower bound by the same ideas  as in the case $\alpha >1/3$,
but we have to improve the definition of satisfactory points.
Let  $\M$ a set of  part of  $V(A) $  and let 
$\epsilon >0$  and  $y>0$. Given $U\subset V(A\wr_{\A_g}B)$ and $f$ a
configuration of $U$, we say that the configuration $f$ is  $(1-\epsilon,y)_{\M}$
satisfactory if there  exists $M\in \M$ such that $M'\subset M$  and  
$(1-\epsilon)|M| \leq |M'|$, where  
$M'=\{a\in V(A);\ \underset{\phi(a)}{\dim} f \geq y \}$.\\
Then the proof falls into  $3$ steps.
 \begin{enumerate}
\item Let  $U\subset V(D_F )$ such that  
$\frac{|\partial_{D_F} U|}{|U|}\leq \frac{1}{k}  $.   (**)
\item For  $W\subset V( D_F )$, we call $W_c=\{f; \ \exists a \in V(A)
\ (a,f)\in W \}$.  By the same way as in the proof of propostion
\ref{existence} in the case $\alpha >1/3$, we prove that there exists
$\epsilon >0$ such that for all  $U$  verifying  (**),
  there exists  $W\subset U$ such that all $f$  of  $W_c$  is   
$(1-\epsilon, Fol_B(k)/30)_{\M}$ satisfactory, with  
$$\M=\{D\subset V(A);\  
\frac{|\partial_A D|}{|D|}\leq\frac{1}{k}\}.$$
This result is analogous to  lemma \ref{neuds} et
\ref{soushypergraphes}  is proved in the next section \ref{complement}.
\item  Take  now  $f\in W_c$,  there exists $M\in \M$ such that,
\begin{eqnarray*}
\begin{cases}
M'=\{ a\in V(A);\ \underset{\phi(a)}{\dim} f \geq Fol_B(k)/30 \} \subset M \\
\text{and} \\
|M'| \geq (1-\epsilon)|M| .
\end{cases}
 \end{eqnarray*}

Lemma \ref{noel} apply with $\delta=1-\epsilon, \ M=\Omega$ and 
$M'=\Omega_{\delta}$.  We deduce that for all $f$ in $W_c,$ we can change the value of the
configuration $f$  in at least  $ c_1\frac{1-\epsilon}{2K} g(c_2 
Fol_A(k))$ points in $Fol_B(k)/30$ ways by staying in  $W_c$. Then we conclude by the following lemma:

\begin{lem}
\label{choixsuccessifmain} Let  $Y >0 $ and   $X>0$. Let  $\A$ a non
empty set  of configurations, such that for all configurations of $\A$,
there exists at least   $Y$ points where we can change the value of the 
configuration in  $X$ way without leaving    $\A$.  Then  : $|\A| \geq X^Y.$\\
\\
ie:\\
 $ ( \forall f \in \A \ \exists a_1,a_2,...,a_Y \in A \ such \ that  \ \ g \in \A) 
 \Longrightarrow  |\A|\geq X^Y,$
\\
where  $g$  is defined from  $f$  by : 
 $g (x)=
\begin{cases}
 \;  f(x) & \text{if $x \neq a_{i_0}, $} \\
  \; X \ \text{possibilities } \  for  \  g(a_{i_0}) & \text{if $x=a_{i_0}.$  }
\end{cases}$
\end{lem}

\begin{proof}
We proceed by induction on   $Y$.\\
If $Y=1$, it is  exact.\\
Suppose $Y\geq 1$  and consider a point $x_0$  in the base such that
there exists  $X$ distinct configurations   $f_1,..., f_X \in A $ such
that  
 $\forall y \neq x_0 \ f_1(y)=f_2(y)=...=f_X(y) .$  \\
 For all   $i=1...X$,  let   $\ \A_i = \{f\in \A ; f(x_0)=f_i(x_0)  
 \}, $ which are not empty.\\
$\A=\underset{i=1...X}{ \dot{\bigcup} }A_i$  and the   $A_i$  satisfy
induction hypothesis with constant  $Y-1$. \\
So,  $|\A| = \underset{i=1...X} {\sum} |A_i| \geq X.X^{Y-1} =X^Y.$
\end{proof}

Finally, lemma  
\ref{choixsuccessifmain}   gives, $$|U|\geq |W_c| \geq 
(\frac{Fol_B(k)}{30} )^{ c_1' g(c_2 Fol_A(k) ) } \succeq e^{g(k)},$$
since first  $Fol_B(n)=2$ and secondly  $ Fol_A(k) = 2k$.
 
\end{enumerate}

\subsubsection{ proof of (ii) of proposition \ref{existence}}

We follow idea of   the case  $\alpha \geq 1/3$.\\
\begin{enumerate}
\item Let  $d_0=(0,f_0)$  where  $f_0$  is the
configuration which is null every where. Let  $X_n=(K_n,g_n)$  be 
the random walk on $D_F$ defined above.   $X$ starts
from $d_0 $ and jumps uniformly on the set of points
formed by the point where the walk is and its
neighbor.  On this generalized wreath product, this
walk is still reversible for the uniform measure
since the number of neighbor in $D_F$ is constant,
equal to  $4$. Now write: 
\begin{eqnarray*}
p^{D_F}_{2n}(d_0,d_0) &=& \underset{z}{\sum} p_n^{D_F}(d_0,z)
p_n^{D_F}(z,d_0)\\
&\geq& \underset{z \in G}{\sum}  p^{D_F}_n(d_0,z) ^2  \\
&\geq& \frac{1}{|G|}[ \underset{z \in G}{\sum} p^{D_F}_n(d_0,z)]^2\\
&\geq& \frac{1}{|G|}[  \P^{D_F}_{d_0} (X_n \in G)]^2,
\end{eqnarray*}
where  $G$  is some finite set of  $V(D_F).$\\ \\
\item Take  $G=G_r=\{  (a,f) \ ; \ | a| \leq r \ and \ supp(f)\subset 
\phi([|-r,r|]) \}. $ \\
By the  structure of edges on  $D_F$, we have : 
\begin{eqnarray*}
\P^{D_F}_{d_0}(X_n \in G_r) &\geq& \P^{D_F}_{d_0} 
( \forall i \in [|0,n|] \ |K_i|\leq r  ) \\
&\geq&  \P_0^K ( \forall i \in [|0,n|] \ |K_i|\leq r  ),
\end{eqnarray*}
where  $\P_0^K$ is the law of  $(K_i)_i$ which is still
a random walk with transitions probability which can
be represented by : \\
$$\includegraphics*[width=6cm]{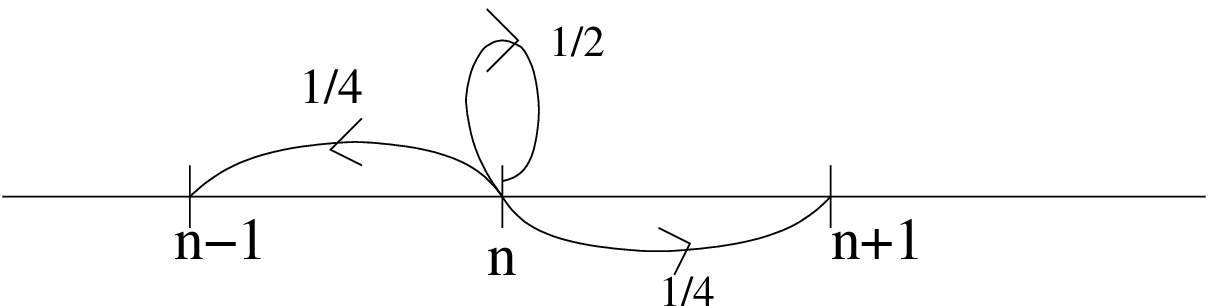}$$
\item  Now we  have to find a lower bound for $\P_0^K (
\forall i \in [|0,n|] \ |K_i|\leq r  )$.      It is
not sufficient to use  $\P_0^K (
\forall i \in [|0,n|] \ |K_i|\leq r  )  \geq 
e^{-c(n/r^2+r)}$ as in the case $\alpha>1/3$, because 
$\beta=\frac{2\alpha}{1-\alpha} <1$ ( see step D of
this proof).  However we can prove that : 
$$\exists c>0, \forall n\geq 0 \  \ \     
\P_0^K ( \forall i \in [|0,n|] \ |K_i|\leq r  ) \geq   e^{-cn/r^2}.$$
One can find this result in the lemma 7.4.3 of
\cite{SPbook}. It is known for a simple random walk
on  $\Z^d$ and we can deduce it in this particular
case with a coupling.
Consider  $K_i'$  which takes values in  $\Z^2$.
$K_i'$  
follows the  horizontal jumps of  $K_i$ if  $K_i$ 
moves and jumps uniformly on its 2 vertical neighbors
if  $K_i$  stays at its place.   On the first hand we
have 
$\{ \sup_{0\leq i \leq n} |K_i'| \leq r\}\subset
\{\sup_{0\leq i \leq n} |K_i| \leq r\}$  and on other
hand  $K_i'$  is a simple random walk on  $\Z^2$. 
(For $x=(a,b)\in \Z^2$, we note $|x|:= \max(a,b )$.)
Then the result for  $K_i$ in  
$\Z$  follows from the result for  $K_i'$ in $\Z^2$.
  \item We can end up the proof.  From  $|G_r|=
  (2r+1) 2^{N^{\A_g}(-r,r) } \preceq e^{ r^{\beta}}$,
  we deduce there  exists $c>0$ such that :
$$   p^{D_F}_{2n}(d_0,d_0)  \geq  e^{ -c( \frac{n}{r^2} +  r^ {\beta }   )}  .$$
But the  function  $r \mapsto \frac{n}{r^2}+r^{\beta
}$, is  minimal for  $r$ like $n^{\frac{1} {\beta+2} }$.\\
So there exists  $c>0$  such that:
$$p^{D_F}_{2n}(0,0) \geq e^{-c n^{\frac{\beta}{\beta+2} } }=e^{-cn^{\alpha}} .$$

\end{enumerate}

\subsection {Complement on satisfactory points}\label{complement}
In this section we improve the notion of satisfactory
point used in subsection \ref{d}  which is more
abstract  that the notion introduced in subsection
\ref{f}.  The reasons of this improvement will be
explain in the next.
 
We still consider a wreath product  $A\wr B$  of
two graphs $A$ and  $B$  or a generalized wreath
product $A\wr_{\A} B$ associated to some partition
$\A$.  We take    $U\subset V(A\wr B)$ and as before to each  $U$ we associate an hypergraph
$K_U$ and its one dimensional skelelton  $\Gamma_U$
with weight $w$,
built as the same way that in section \ref{f}.

Let  $\epsilon >0$ and  $a \geq 0$. Let  $\M$ a set
of parts of  $V(A)$.  To light the way of   this 
definition and to link it  with the old definition of
satisfactory points (section \ref{f}),
one can think to take for  $\M$  set of the form 
$\{ D\subset V(A); \   \frac{|\partial_A D|}{|D|} \leq
\frac{1}{k} \}$. 
\begin{enumerate} [$\bullet$]
\item  A configuration $f$ of $V(K_U)$ is said  $(1-\epsilon,a)_{\M}$
satisfactory if :
\begin{eqnarray} \label{newdef}
\text{there exists } \ M\in \M \ \text{such that }
\begin{cases}
 \;  M'\subset M \\
 \; \text{and}\\
\;  (1-\epsilon) |M| < |M'|
\end{cases} 
\end{eqnarray}
where  $ \ M'=\{m\in V(A);\ \underset{m}{dim} f \geq a \}.$\\
Once again, we denote by $S_U(1-\epsilon,a)_{\M}$ (or
$S(1-\epsilon,a)_{\M}$ ) the set of satisfactory
configurations.
\item Otherwise  $f$ is not satisfactory and we  note 
$NS(1-\epsilon,a)_{\M}$ ) the set of not satisfactory
configurations.
\item If $\Gamma'$  is a subgraph of  $\Gamma_U$,  we
say that  $f$  is $S(1-\epsilon,a)_{\M}$  in respect
to  $\Gamma'$  if  
$f$ satisfies the same condition  as in (\ref{newdef}) but  where 
dimension of  $f$  is counted only with edges in  $\Gamma'$.
More precisely :\\
$\underset{m,\Gamma'}{dim} f = 
 \#\{ g; \ (f,g)\in E(\Gamma') \ \text{and } \ 
  (x,g) \in U \ and \ \forall y \neq x \ f(y)=g(y)    \}  $.\\
\item  An edge of $\Gamma_U $ is said   $(1-\epsilon ,a)_{M}$ 
satisfactory if it joins two  
$(1-\epsilon,a)_{\M}$ satisfactory configurations,
otherwise it is said  $(1-\epsilon,a)_{\M}$ not 
satisfactory. As before we denote by
$S^e(1-\epsilon;a)_{\M}$ [resp $NS^e(1-\epsilon,a)_{\M}$]
the set of satifactory edges  
[resp not satisfactory ]. \\
\item  A point $u =(x,f)\in U$ is said  $(1-\epsilon,a)_{\M}$ 
satisfactory 
[resp $(1-\epsilon,a)_{\M}$ not satisfactory ] if  $f \in S(1-\epsilon,a)_{\M}$ 
[resp $ NS(1-\epsilon,a)_{\M} $]. We denote by 
$S^p(1-\epsilon,a)_{\M}$ and 
$NS^p(1-\epsilon,a)$  for the set of points which are
(or are not )  satisfactory.  
\\
\item We keep the same defintion for good points, 
$u=(x,f)\in U$  is said   $ a-good$ if  
$ \underset{x}{dim} f \geq a$ otherwise it is said  
$a-bad.$\\
\end{enumerate}

The interest of this new definition of satisfactory
points is the following. Consider a set  $U_c $  of 
$(\lambda,b)-satisfactory $ configurations. With the
"old" definition we know that we can change the value
of  $f$  in at least  $b$ points  in $\lambda_x$ 
ways (at point $x$)  without leaving  $U_c$ but we
do not know exactly where are these $b$ points 
whereas   with the "new" definition, for a set  
$U_c$  of   $S(1-\epsilon,a)_{\M}$ configurations, we
know  that  we can change the value of  $f$ in at
least 
 $(1-\epsilon) \underset{M\in \M}{\min} |M|$ points
 in $a$  ways without leaving $U_c$   et moreover we
 know that these points are contained     in some 
 $M\in\M$.  This would be useful for our generalized
 wreath products since this property concentrate  points
 where we can change value of $f$. By the properties  of
 partition, it remains only to get lower bound of $\#\phi(M)$.\\

Let  $U\in V(A\wr B)$ such that
$\frac{|\partial_{A\wr B} U|}{|U|} \leq \frac{1}{1000k}$, 
the two following lemmas are similar to  lemmas
\ref{neuds} et \ref{soushypergraphes}.
\begin{lem} \label{neudnew}
Let $\M=\{ D\subset V(A);  \  
\frac{|\partial_A D|}{|D|} \leq
\frac{1}{k} \}$ then we have :
\begin{enumerate} [$(i)$]
\item $\frac{ \#\{ u\in U; \ u \ is \ Fol_B(k)-
\text{bad } \}}
{ |U|} \leq \frac{1}{1000},$ \\
\item  there exists  $\epsilon >0$  such that 
$\ \frac{ \#\{u\in U;\ (1-\epsilon, 
Fol_B(k)/3)_{\M}-\text{not  satisfactory}      \}}{ |U|} \leq \frac{1}{500}. $
\end{enumerate}
\end{lem}

\begin{proof} $\ $ \\
For (i), it is the same argument that in part 
(i)  of  lemma \ref{neuds}.\\
For (ii) let,  
\begin{eqnarray*}Neud &=&\{  u \in U ;\  u \in
NS^p(1-\epsilon, \frac{Fol_B(k) } {3} )_{\M} \  \} \\
&=&\{ u =(x,f) \in U ; f \in \ NS(1-\epsilon,\frac{Fol_B(k)}{3}  )\},
\end{eqnarray*}
and let:  
$$Neud(f) = \{ (x,f) ; (x,f)\in U  \}.$$  Note
that   $p(Neud(f))=\{x; (x,f)\in U\} .$ \\
  For  $F$ a set of  configurations, let   
   $$Neud(F)= \underset{f\in F}{\cup} Neud(f).
   $$ Note the union is disjointed. \\
   \\
  Take  now $f \in NS(1-\epsilon,\frac{Fol_B(k) } {3})_{\M} $, 
    and consider the set  $p(Neud(f))$.\\
  Two cases appear. Either $p(Neud(f))$ gives
  a large part of boundary in  "base" either not
  and this case by assumptions on $f$ we
  will prove that $p(Neud(f))$ gives
   boundary in "configurations" \\

\underline{First  case} : $f\in F_1 := \{  f\in NS(1-\epsilon, 
\frac{Fol_B(k)}{3})_{\M} ;   
\frac{\#\partial_A \ p(Neud(f)) }{\# p(Neud(f))} > \frac{1}{k} \} .$\\

The application$ \underset{f\in
F_1}{\dot{\bigcup}}\partial_A \ p(Neud(f))
\longrightarrow  \partial_{A\wr B} U  \ $ is 
injective. \\
\hspace*{4.8cm} $       (x,y) \longmapsto \Bigl(\left(x,f\right);\left(y,f\right)  \Bigr)$ \\

So, we can write :
    \begin{equation} \label{cas1mainnew}
 |\partial_{A\wr B} U  | \geq 
 \underset{f\in F_1}{\sum} | \partial_A \ 
p(Neud( f) )| \geq  \frac{1}{k} \underset{f\in F_1}{\sum} |
p(Neud(f))|
\geq  \frac{1}{k}|Neud (F_1)|  . 
\end{equation}
\\

\underline{Second case} : $f\in F_2 := \{  f\in NS(1-\epsilon, 
\frac{Fol_B(k) } {3})_{\M}  ;  
 \frac{\#\partial_A \ p(Neud(f)) }{\# p(Neud(f))} \leq \frac{1}{k} \} .$\\
\\
Since $f\in NS(1-\epsilon,\frac{Fol_B(k) }
{3})_{\M}   $  we have   :
\begin{eqnarray} \label{M'}
\text{for all  } M \in \M \ \ 
\begin{cases}
\;  \exists\  m'\in M'-M,  \\
\; \text{or } \\
\; |M'| \leq (1-\epsilon) |M|,
\end{cases}
\end{eqnarray}
where $M'$  stands for  $\{m\in V(A);\ 
\underset{ m}{\dim}f \geq  \frac{Fol_B(k) } {3} \}.$\\
Choose $M=p(Neud(f))$ since  $f\in F_2$ we have 
$ M\in \M$ and  $M'\subset M$. So it is the
second item of assertion (\ref{M'})  which is
satisfied. ie :
 $|M'|\leq (1-\epsilon) |M|$. So,
$$ \# \{ x\in p(Neud(f)) ; \ \underset{x}{dim} f \geq \frac{Fol_B(k)}{3} \} < 
 (1-\epsilon) |M|=(1-\epsilon) |Neud(f)| .$$ 
 (We have used that   $|p(Neud(f)|=|Neud(f)|.$)\\

So $$ \# \{ x\in p(Neud(f)) ; \ \underset{x}{dim} f < \frac{ Fol_B(k)}{3} \}
 \geq  \epsilon |Neud(f)|     $$
   \begin{eqnarray} \label{P_f} ie :   |P_f| \geq \epsilon  |Neud(f)|,
   \end{eqnarray}
with   $P_f=\{ x\in p(Neud(f) ; \ \underset{x}{dim} f  < \frac{Fol_B(k)}{3}  \}. $ \\
To each point of  $P_f$ ( for  $f$  in  $F_2$), 
we can associate in an injective way a point of
the boundary  (in configuration ) of $U$.
Indeed, as before :\\
for $ x \in P_f $ and   $f \in Neud(F_2)$, we
have :
$$|\tilde{P}_{x,f}| \leq \frac{Fol_B(k)}{3} <  Fol_B(k).$$ 
where  $\tilde{P}_{x,f} =\{ g(x) ; \ (x,g) \in U
\ and \ \forall y \neq x \  g(y)=f(y)\}$.\\
 Thus,   $$ |\partial_B \tilde{P}_{x,f}| > \frac{1}{k } |\tilde{P}_{x,f}| \geq 0,$$ 
  and then  $$ |\partial_B \tilde{P}_{x,f}| \geq 1.$$ \\
   Finally, 
 \begin{eqnarray*}
|\partial_{A\wr B} U| &\geq &
\underset{x\in P_f, f\in F_2}{\sum}
|\partial_B \tilde{P}_{x,f}| \\
&\geq & \underset{f\in F_2}{\sum} \epsilon
|Neud(f)|     \hspace*{1cm} \text{by (\ref{P_f}}), \\
&\geq &  \epsilon |Neud(F_2)|\\
&\geq &  \frac{1}{k}  |Neud(F_2)|   \text{
by choising   $\epsilon < 1/k$. }
\end{eqnarray*}
By adding (\ref{cas1mainnew})  and this last
inequality and using the fact that  
$\frac{|\partial_{A\wr B} U| } {| U| } <
\frac{1}{1000k},$  we get : \\
$$ \frac{|Neud|}{|U|} < \frac{1}{500}. $$
\end{proof}

\begin{lem} \label{topchiantsoushyper}

Let  $\epsilon>0 $ and $x>0$. Consider  $ \Gamma_U $
 the one dimensional skeleton with weight $w$, 
 constructed from   $K_U$. 
Assume that   $E(\Gamma_U ) \neq \emptyset$  and  $\forall f \in K_U  \ 
\underset{x}{dim} f \geq a $ and  $\M$  does not contain the
empty set.  If we have :\\
$$ \frac{ \underset{e\in  NS_U^e( 1-\epsilon ,a )_{\M} }{\sum } w(e) }
{\underset{e\in E(\Gamma_U )}{\sum } w(e)}<1/2,
$$\\
then, there exists  a not empty subgraph   $\Gamma'$
of $\Gamma_U $  such that all edges are  $S_U( 1- \frac{9+\epsilon}{10},\frac{a}{10})_{\M}$ 
satisfactory in respect to  $\Gamma'$.
\end{lem}
\begin{proof}

In the graph  $\Gamma_U$,  we remove all points   
$NS^p_U( 1- \frac{9+\epsilon}{10},\frac{a}{10})_{\M}$
and adjacents edges. After this first step, it may
appear  some new points 
$NS^P_{U_1}( 1- \frac{9+\epsilon}{10},\frac{a}{10})_{\M}$, 
where  $U_1=U- NS^p_U( 1- \frac{9+\epsilon}{10},\frac{a}{10})_{\M}  $. \\
We remove  again all adjacent edges and points and
we iterate this process.\\
Let  $U_i$ the set of vertices staying at step  $i$. 
$$\left \lbrace
\begin{array}{l}
U_0=U,\\
\mathrm{for} \ i\geq 1 \ \ U_{i+1}=U_i -NS_{U_i}^p
( 1- \frac{9+\epsilon}{10},\frac{a}{10})_{\M}.
\end{array}
\right.$$
it is sufficient to prove that this  process ends up
before the graph becomes empty.\\
Let $C_1= \underset{e\in NS_U( 1-\epsilon,a)_{\M}}{\sum }w(e) $ ,
$\ \ \ C_2=   \underset{ 
\underset{ \text{ at the end of the process} }
{e\in S_U^e(1-\epsilon,a)_{\M} ; e \text{ removed } }  
}{\sum } w(e) ,$ \\
 and
  $$ C_0= \underset{
\underset{\text{ at the end of the process}  }{e\in
E(\Gamma_U) ) ; e \text{ removed  } }
} 
 {\sum }
 w(e)    . $$
If we show that $ C_2 \leq C_1$,  the result is
proved since: $$C_0\leq C_1 +C_2 \leq 2 C_1<
  \underset{e\in E(\Gamma_U ) }{\sum} w(e).$$ \\
That would mean that it stays at least one point not
removed. 
 ie: $\exists k_0\in \N  $  such that all points of
 the graph get at step   $k$,  are  
 $S^p_{U_{k_0}}( 1-
 \frac{9+\epsilon}{10},\frac{a}{10})_{\M} $,  so  
 $ S^p_U( 1- \frac{9+\epsilon}{10},\frac{a}{10})_{\M}.$    \\
 \\
To see this, let us introduce an  orientation on
removed edges :  if  
$L$ and  $Q$  are points of the graph,  we   orient
the edge from   $L$  to  $Q$ if  $L$  s removed
before  $Q$ otherwise  we choose an arbitrary
orientation.  We note   
$\underset{\downarrow}{L}$  the set of edges leaving
the  point $L$ and $\underset{\uparrow}{L}$  for the
set of edge ending in $L$ at step  $0$.

\begin{souslem}
\label{soustrucchiantsmain} Let  $k \in \N$  and let
$L $  be a  point of the   graph 
$\Gamma_U$ (satisfying assumptions of lemma \ref{topchiantsoushyper}) 
removed after  $k+1$ steps.
Assume that   $L$  is initially $ S^p_{U}( 1- \epsilon ,a)_{\M}
$, then
$$ \underset{e \in \underset{\downarrow}{L} }{\sum} w(e)   \leq \frac{1}{2} \underset{e \in  
\underset{\uparrow}{L}}{\sum} w(e)   .$$
\end{souslem}

\begin{proof}
It would be useful to notice that  for a multidimensional edge  $e$, the
sum of the weight (in the skeleton)  of edges coming from  $e$ and
adjacent to a point, is equal to $ 1$. This is implied by our choice of
the weight. \\
\begin{enumerate}
\item
Let now  $N_0$ be the number of   multidimensional
edges at step 
  $0$.  Since   $L$ is  initially 
  $S^p_U(1-\epsilon  ,a)_{\M},\ $  there exists 
  $M_0\in\M$ such that  
   $(1-\epsilon)|M_0|$    multidimensional edges are
   attached to   $L$. So,
\begin{eqnarray} \label{N_0n} N_0 \geq (1-\epsilon )|M_0| .
\end{eqnarray}
Besides notice that  :
$$  \underset{ \underset {e \ contains  \ L}{ e\in E(\Gamma_U ) } }{\sum} w(e)  =N_0 .$$
Let  :\\
$ \underset{\downarrow}{L_1}  = \{ e\in     
\underset{\downarrow}{L}, \ e \text{ coming from a  
multidimensionnal edge  of } \ K_{U_k} ,\\
\hspace*{9cm}  \text{of} \ dim\geq a/10   \}, $  \\
and\\
$ \underset{\downarrow}{L_2}  = \{ e\in     
\underset{\downarrow}{L}, \ e \text{ coming from a 
multidimensionnal edge} \ K_{U_k} ,\\
\hspace*{9cm} \text{of} \ dim < a/10   \}. $ \\

We have $\underset{\downarrow}{L} = \underset{\downarrow}{L_1} \cup 
\underset{\downarrow}{L_2} $, because edges of   $  
\underset{\downarrow}{L}$ correspond to  edges leaving  
$L$ at step $k$.
\\
\item Since  $L$ becomes   $NS^p_{U_k}( 1- \frac{9+\epsilon}{10},\frac{a}{10} )_{\M} $,
we have:
\begin{eqnarray*}
\text{for all  } \ M \ \text{in } \ \M
\begin{cases}
M''\not\subset M \\
\text{or} \\
|M''|\leq ( 1- \frac{9+\epsilon}{10}) |M|
\end{cases}
\end{eqnarray*} 
where $M''=\{ m\in V(A);\ \underset{m,U_k}{dim}  L \geq \frac{a}{10}  \}$.\\
\\
Take $M=M_0$, observe that  $M''\subset M_0$   so 
that implies   
$|M''|\leq ( 1- \frac{9+\epsilon}{10}) |M_0|$.
Finally   $L$ has less than   
$ (1- \frac{9+\epsilon}{10}) |M_0| $  
multidimensional edges of dimension at least 
$a/10$. call them  $f_1, ...,f_{q}, $ with  
$q < (1- \frac{9+\epsilon}{10}) |M_0|   .$    
\begin{eqnarray} \label{poid1n}\underset{e\in  \underset{\downarrow}{L_1} }{\sum} w(e)  = 
 \underset{k=1..q}{\sum} \ \ 
    \underbrace {\underset{\underset{ coming  \ from \  f_{k}}{e} }{\sum} w(e) }_{\leq 1} \leq q  .
    \end{eqnarray}
    (Initially this last sum was equal to  $1$, but
    after removing some edges, this sum is less than
    $1$.)\\
    \\
 Besides, call  $g_1,...,g_{h}$ the other 
 multidimensional edges of  dimension strictly less
 than  $a/10,$ attached to  $L$ at step  $k$,  with  
 $h\leq N_0-q$.\\
 For all   $k=1...h,$
 \begin{eqnarray} \label{poid2n}
 \underset{\underset{coming \ from \   g_{k}}{e} }{\sum} w(e)  \leq 
 \frac{1}{a} \frac{a}{10} \leq \frac{1}{10} . 
 \end{eqnarray}  
 (Indeed, first all point have initially  dimension
 at least   $a$ so we deduce  
  $\forall e \in E( \Gamma_{U_i} )  \ 
 w(e) \leq 1/a$ and secondly an edge of dimension
 less than $a/10$  gives less than  $a/10$
edges attached to one point, in the skeleton. ) \\
 \\ 
\item Finally with  (\ref{poid1n}) and 
(\ref{poid2n}), we get :
\begin{eqnarray*} 
\underset{e\in  \underset{\downarrow}{L} }{\sum} w(e)  &= &
\underset{e\in  \underset{\downarrow}{L_1} }{\sum} w(e)  +
\underset{e\in  \underset{\downarrow}{L_2} }{\sum} w(e) \\
&\leq&  q + (N_0-  q   )\frac{ 1}{10} \\
&=&  \frac{1}{10} N_0+ \frac{9}{10} q\\
&=& \frac{19}{100} N_0.
\end{eqnarray*}
(
$q <  ( 1- \frac{9+\epsilon}{10}) |M_0| \leq  N_0 \  
\frac{1- \frac{9+\epsilon}{10}}{1-\epsilon}
\leq \frac{N_0}{10} \ \ \ $  by  (\ref{N_0n}).)\\
\\
So,
$$ \underset{e\in  \underset{\downarrow}{L} }{\sum} w(e)  \leq \frac{19}{100} N_0  \ \ et\ \
 \underset{e\in \underset{\uparrow}{L}}{\sum} w(e)  \geq  N_0 - \frac{19}{100}N_0=\frac{81}{100}N_0    .$$
And then,  
$$
\underset{e \in \underset{\downarrow}{L} }{\sum} w(e)   \leq \frac{19}{81}\underset{e \in  
\underset{\uparrow}{L}}{\sum} w(e) 
\leq \frac{1}{2}  \underset{e \in  
\underset{\uparrow}{L}}{\sum} w(e) .
$$

\end{enumerate}
\end{proof}

The proof ends up by the same way as proposition
\ref{soushypergraphes}, let:\\
$D_1=\{ \text{vertices   \ removed \ at  \ step} \  1 \},$ \\
and for   $i\geq 2$ \\
$D_i=\{ \text{vertices } \ S^p_{U}(1-\epsilon,a) \
\text{removed \ at  \ step}\  i \},$ \\
$F_i= \{ \text{edges  \ between }\ D_i \ and \ D_{i-1}  \},$ \\
$ F_i'= \{ \text{edges \ leaving  }\ D_{i-1}  \}     $.\\
\\
Notice that $F_i \subset F_i'$ and that edges of   
$F_i'$  are removed.\\
\includegraphics*[width=11cm]{flot.eps}
By sublemma  \ref{soustrucchiantsmain}  applied in
each point of  $D_i$ in the graph get at step   $i-2$.
(Each  point of   $D_i$  is at this moment, at least
$S(1- \epsilon ,a)_{\M}$.) We get   :
$$\forall i \geq 2 \ \  \underset{e\in F_{i+1}'}{\sum} w(e)  \leq \frac{1}{2} \underset{e\in F_i}{\sum}w(e) . $$

so, $$ \ \  \underset{e\in F_{i+1}'}{\sum} w(e) \leq (\frac{1}{2})^{i-1}  \underset{e\in F_2}{\sum} w(e) . $$
(We have used that   $F_i \subset F_i'.$)\\
Thus,
\begin{eqnarray*}  \ \underset {e\in   \underset{i\geq 3}{\cup} F_i'}{\sum} w(e)   &\leq&
(\underset{i\geq 1} {\sum} (\frac{1}{2})^i)  \underset{e\in  F_2}{\sum} w(e) \\
&=&  \underset{e\in  F_2}{\sum} w(e) .
\end{eqnarray*}
Now, an edge of   $F_2$ is  $NS_U^e( 1- \epsilon,a)$
 because if this edge  was  $S_U^e(1-\epsilon,a)$,
 this edge would have linked two  points 
$S_U^p(1-\epsilon,a)$ and in particular,    points 
of  $D_1$  would have been   $S_U^p(1-\epsilon,a)$,
so $S_U^p(1- \frac{9+\epsilon}{10} ,a/10)$ 
 and so  would not have   removed.
Thus :
 $$\underset{e\in F_2}{\sum} w(e)  \leq \underset{e\in NS^e(a,b)}{\sum} w(e)=C_1  .$$
Besides, all removed edge  $S_U^e(1-\epsilon,a)$  is
in some  $F_i'$ with  $i\geq 3$,  so 
$$C_2=\underset{ 
\underset{e \in S_U^e(1-\epsilon,a) }
{ e \text{ \ removed \ at  \ the   \ of  \ the \    process} }}
{\sum}
w(e)      
\leq  \underset{e\in \underset{i\geq 3}{\cup} F_i'}{\sum} w(e) .$$ \\
That  ends  the proof.

\end{proof}

Now we are able to explain the fact that we used in
section  \ref{ipropdur} in order to  prove  the
lower bound of  $Fol_{D_F}$.
We recall that  $U\subset V(D_F)$  is such that 
$\frac{|\partial_{A\wr B} U|}{|U|}\leq\frac{1}{k}.$
Let  $\tilde{K}$ be the sub hypergraph of $K_U$ which
contains only $Fol_B(k)/3-$good points. As in the
proof of (i) in the case $\alpha>1/3$ of propostion
\ref{existence}, we prove by using
lemma \ref{neudnew} that there exists $\theta<1/2$
such that, 
$$ \frac{ \underset{
\underset{ e\in E(\Gamma_{\tilde{K} } ) }
{
e\in  NS_U^e( 1-\epsilon ,Fol_B(k)/3 )_{\M} 
}
}{\sum } w(e) }
{\underset{e\in E(\Gamma_{\tilde{K} })}{\sum }
w(e)}<\theta,$$  for some  
$\epsilon>0$  and $\M=\{D\subset V(A);\
\frac{|\partial_A D|}{|D|}\leq\frac{1}{k}\}$.\\
Lemma \ref{topchiantsoushyper} gives us a sub  graph
where all edges are  $S^e(1-\delta, Fol_B(k)
/30)_{\M}$ for 
$\delta= 1-\frac{9+\epsilon}{10}$. By definition of
satisfactory points, this proves the
fact that we have used.

\section{Applications: study of  some functionals}
\label{appli}

\subsection{Kind of problems, case of  the lattice $\Z^d$}
Recall that  for $G$  a graph and $X$ is a simple random walk on
$G$, we note  $ L_{x,n}=\#\{k\in[0;n];\ X_k=x\}$. The question is to
estimate
functional of type 
\begin{eqnarray}\label{functional}E_0^{\omega} (
e^{
- \lambda \underset{z; L_{z;n}>0  } {\sum}   F(L_{z;n},z)  
} \ ),  \end{eqnarray}
where $F$ is a two variables  non negative  function. The method 
developped here is due to Erschler and can be applied on general
graph $G$ provided the  isoperimetric profile on the graph $G$ is known
and the  function $F$ has some "good" properties.  
For the case of the   simple random walk on $\Z^d$, in
\cite{ershdur} it is proved that 
\begin{eqnarray} \label{ext1Z}\forall \alpha \in [0,1] \ \ \  \E_0^{\omega} (
e^{
- \lambda \underset{z; L_{z;n}>0  } {\sum} L_{z;n}^{\alpha} 
} ) \approx  e^{-n^{\eta}},   \end{eqnarray} 
\begin{eqnarray} \label{ext2Z}     \forall \alpha \geq 1/2 \ \
\E_0^{\omega}( \underset{ z;L_{z;n} >0}{\prod }  L_{z;n}^{-\alpha}   ) 
 \approx  e^{-n^{\frac{d}{d+2}} 
 ln(n)^{\frac{2}{d+2} }},\end{eqnarray}
where  $\eta= \frac{ d+\alpha (2-d) }{ 2+d(1-\alpha) }.$ This
section is devoted to extend these estimates to an infinte cluster
of the percolation model. 
\subsection{In an infinite cluster of the percolation model}
\subsubsection{Percolation context } $\ $ \\
Consider the graph $\L^d=(\Z^d,E_d)$ where $E_d$ are the couple of
points of $\Z^d$ at distance $1$ for the $N_1$ norm. Now pick a
number $p\in ]0,1[$. Each edge is kept  [resp
removed ] with  probability $p$  [resp $1-p$]  in an independant way.  
We get a graph $\omega$ and we call $\C $ the connected component
that contains the origin and $\C_n$ the connected component of
$\C\cap [-n,n]^d$ that contains the origin. 

We still use the notation $\omega$ for  the application  $E_d\rightarrow \{0,1\}$
such that $\omega(e)= 0$ if $e$ is a  removed edge and $1$ 
otherwise.  Let $Q$ be the  probability measure under which the
variable $(\omega(e), e\in E_d) $ are  Bernouilli(p) independent variables. If  $p$  is larger than some critical value $p_c$, the 
$Q$ probability that
$\C$ is infinite, is strictly positive  and so we can
work on the event  $\{\#\C=+\infty\}$.\\

We denote  by $\C^g$ the graph such that $V(\C^g)=\C$ and $E(\C^g)=\{
(x,y)\in E_d;\ \omega(x,y)=1 \}$ and $\C_n^g$ the graph such that
$V(\C_n^g)=\C$ and $E(\C_n^g)=\{
(x,y)\in E_d;\ x,y\in \C_n\text{ and } \omega(x,y)=1  \}$

From now on and until the end, $X$ will design the simple random
walk on the graph $\C^g$. We are going to prove estimate
(\ref{ext1Z}) and (\ref{ext2Z}) for the walk $X$.
\subsubsection{Sketch plan} $\ $ \\

Let   $(B_x )_{x\in \C}$  be a  family of  graphs and let  $0_x$ 
an arbitrary point in each  $B_x$  that we  call the  origin. 
For all $x\in \C$, consider the random walk  $(Y^x_n)_n$ on 
$B_x$  starting from point  $0_x$,   and jumping uniformly on
the set of points formed by the point where the walk is and its
neighbors.  Let  $\P_{0_x}^{B_x}$ be the law of  $(Y^x_n)_n$.\\
Transition kernels of  $Y^x$  satisfy : 
$$p^{B_x}(a,b) = \frac{ 1}{\nu_x(a)+1} (    1_{\{ a=b\}} +1_{\{ (a,b)\in
E(B_x)\}} ),$$
where  $\nu_x(a) $ stands for the number of neighbors of  $a$
in graph $B_x$.\\
\\
Consider now the graph  \begin{eqnarray} \label{W}W=W_{\C}=\C^g
\wr (B_z)_{z\in\C}. \end{eqnarray}
Let $f_0 $  be the nulle configuration, such that , for all  $x\in
\C$, $f_0(x)=0_x$, and let $o=(0,f_0)$.\
And we look at the  random walk  $( Z_n)_n$ on the graph 
$W_{\C}$ starting from $o$,  defined by the following:
suppose that the walk is at point  $z=(x,f)$, then in one unit
of time the walk makes three independent steps. First, 
the value of $f$ at point $x$   jumps in graph $B_x$ in respect
to the walk  $Y^x$ starting from   $f(x)$. Secondly,  we make
the  walker in $\C$ jump on his neighbors in respect to
uniformly law on his neighbor, so the walker in $\C$
(projection on $\C$ of walk on $W_{\C}$) arrives at  point $
y\in \C$. And thirdly, the value of $f$ at point $y$   jumps in
graph $B_y$ in respect
to the walk  $Y^y$ starting from   $f(y)$. 
\\
Thus, calling  $\tilde{p}$  transitions kernel of  $Z$, we
have: \\
for all  $((a,f);(b,g))\in (V(\C^g\wr B_z))^2$:
\begin{eqnarray} \label{ptilde}
\tilde{p}[(a,f)(b,g) ]=   \frac{ \chi[(a,f),(b,g)]  }
{ \nu(a)\  [\nu_a(f(a)) +1]\  [\nu_b(f(b)) +1] }, 
\end{eqnarray}
 where $\chi[(a,f),(b,g)]$  is equal to  $1 $ if the walk is
 able to jump from  $(a,f)$ to $(b,g)$ and $ 0$ otherwise.\\
More  precisely, 
\begin{eqnarray*}
 \chi[(a,f),(b,g)]=\omega(a,b)\ (\chi_1[(a,f),(b,g)] &+&\chi_2[(a,f),(b,g)]\\
&+& \chi_3[(a,f),(b,g)]+\chi_4[(a,f),(b,g)]),
\end{eqnarray*}
with\\
$\chi_1[(a,f),(b,g)]= 1_{\{\forall x \ f(x)=g(x)\}},\qquad \ \ 
\chi_2[(a,f),(b,g)]= 1_{\{   \underset
{ \forall x\neq a \ f(x)=g(x)}{  \ (f(a) ,g(a)) \in  
E(B_a) }    \}},$ \\
$ \chi_3[(a,f),(b,g)] =1_{\{   \underset
{ \forall x\neq b \ f(x)=g(x)}{  \ (f(b) ,g(b) ) \in  
E(B_b) }    \}} ,\ \ \  
\chi_4[(a,f),(b,g)]= 
1_{\{   \underset
{ \forall x\neq a,b \ f(x)=g(x)}{ \forall  x\in\{a,b\} \ (f(x) ,g(x)) \in  
E(B_x) }    \}}.$ \\
\\
Notice that $\tilde{m}$ defined by,
 \begin{eqnarray} \label{mtilde}
 \tilde{m}(a,f) = \nu(a),
 \end{eqnarray}
is  a reversible measure for the walk $Z$.  We note 
$\tilde{a}$  the following kernels:
 \begin{eqnarray}
 \label{atilde} \tilde{a}(x,y) &= &\tilde{m}(x) \tilde{p}(x,y)
    \end{eqnarray}
 Let  $\tilde{\P}_o^{\omega}$ be the law of $Z$ starting from  $o.$
The key for our  problem is the following interpretation of the
return probability of $Z$: 
\begin{prop} \label{fait0}
$$\tilde{\P}^{\omega}_{o} (Z_{n} =o) =\mathbb{E}_0^{\omega} 
			 ( \underset {x; L_{x;n}>0 }{\prod} \P_{0_x}^{B_x}
			 ( Y^x_{L_{x;n} }  =0_x) 
			 \ \        
1_{       \{  X_{n}=0 \}        }         ).$$
\end{prop}

\begin{proof}: 
\begin{eqnarray*}
\tilde{\P}^{\omega}_{o} (Z_{n} =o) &=& \tilde{\P}^{\omega}_{o} \Bigl(  (X_{n},
f_{n} )=(0,f_0)   \Bigr) \\
                           &=&\underset{ \underset{ k_0=k_{n}=0  }
{(k_0,k_1,...,k_{n} )\in \Z^d} } {\sum}
                                                   \tilde{\P}^{\omega}_o (X_0=k_0,
X_1=k_1,...,X_{n}=k_{n} \  et \ f_{n}=f_0)\\
                          &=& \underset{ \underset{ k_0=k_{n}=0  }
{(k_0,k_1,...,k_{n} )\in \Z^d} } {\sum}
                                                   \tilde{\P}_o^{\omega} (X_0=k_0,
X_1=k_1,...,X_{n}=k_{n} ) \\
&\ & \hspace{4cm} \times \ \ \tilde{\P}_o^{\omega}( f_{n}=f_0 |X_0=k_0,...X_{n}=k_{n}) \\
                          &=&\underset{ \underset{ k_0=k_{n}=0  }
{(k_0,k_1,...,k_{n} )\in \Z^d} } {\sum}
                                                   \P_0^{\omega} (X_0=k_0,
X_1=k_1,...,X_{n}=k_{n} )   \\
& \ & \hspace{5cm} \times \underset {x; L_{x;n}>0 }
{\prod} \P_{0_x}^{B_x} ( Y^x_{L_{x;n}}  =0_x)   \\
                         &=& \mathbb{E}_0^{\omega} 
			 ( \underset {x; L_{x;n}>0 }{\prod} \P_{0_x}^{B_x}
			 ( Y^x_{L_{x;n} }  =0_x) 
			 \ \        
1_{       \{  X_{n}=0 \}        }         ).
\end{eqnarray*}
\end{proof}
In order to estimate  functional such that (\ref{functional}) and
in view of propostion \ref{fait0},  we
have to find graphs $B_x$ such that   for all $m \in \N:$ $$
\P_{0_x}^{B_x}( Y^x_{m }  =0)  \approx e^{-\lambda F(m,x) }.$$
Moreover, since we know that  an isoperimetric inequality with 
volume counted in respect to  measure $m$ and boundary counted
in respect to kernels $\tilde{a}$, gives an  upper
bound of the decay of the probability transitions of walk $Z$, 
in a first time we have to
estimate  the Folner function of $W_{\C} $  and so (by similar results of
section 1, see \cite{ershdur} \cite{ersh} and \cite{rothese}) we should know Folner function
of each $B_x$.\\

The graph formed by the possible jumps of walk $Z$ is not
$W_{\C} =\C\wr (B_z)_{z\in\C}$, so we introduce the graph with same set of points of
$W_{\C}$ but different set of edges. We call it $\C\wr \wr
(B_z)_{z\in\C}$ or shortly $W_{\C}'$ (or $W'$), the graph such that :
\begin{eqnarray} \label{W'} V(W'_{\C})&=& V(W_{\C} )    \
\text{and}, \\
\nonumber
((a,f);(b,g))\in E(W'_{\C}  )  \ &\Longleftrightarrow& \ \chi[ (a,f);(b,g)]=1.
\end{eqnarray}    
Thus, in the graph $W'_{\C}$, the random walk $ Z$ is a nearest
neighbor walk. Properties of $Z$ are linked to geometry of
$W'_{\C}$ but as we will see later $W_{\C} $ and  $W_{\C} '$
are roughly  isometric, so  we  can study isoperimetric profile of $W_{\C}
$.

\subsubsection{ \label{f1} Study of 
$\E_0^{\omega} (e^{
- \lambda 
  {\sum}
L_{z;n}^{\alpha} } ) $
}   
$ \ $ \\
{\bf{ Upper bound}}
Let $\alpha \in ]0,1[$ and 
$\beta=\frac{2\alpha}{1-\alpha}$ and let
$F(x)=e^{x^{\beta}}$.  Let $D_F$ be the graph given by
 proposition \ref{existence}.  We put for all 
 $ x\in \C,\ B_x=D_F$.\\
 
  First we want to obtain a
 lower bound of $Fol_ {         C_n^g \wr\wr D_F}^{
   C^g \wr\wr D_F} (k).$
We proceed in 3 steps:
  \begin{enumerate}[A.]
  \item
  By using general results on wreath product, see
  \cite{ershdur}, \cite{ersh} and \cite{rothese},
  we have  :
  $$Fol_{  \C_n^g \wr D_F }^{   \C^g \wr D_F} (k) \approx 
  (Fol_{D_F} (k) )^{Fol_{\C_n^g}^{\C^g}(k)}.$$

\item By proposition 1.4  of \cite{KL}, we get:\\
for all  $\gamma>0$, there  exists $\beta>0$ such
that for all $c>0$, $\Q$ a.s  for large enough 
$n$, we have :
\begin{eqnarray} \label{Fol^W}
Fol_{\C_n^g \wr D_F}^{\C^g \wr D_F}  (k) \succeq \begin{cases}
 \; F(k)^k & \text{if $k < c n^{\gamma},$} \\
  \; (F( k)^ {\beta k^d} & \text{if  $k\geq   c n^{\gamma}.$}
\end{cases}
\end{eqnarray}\\
\item We want to carry (\ref{Fol^W} )  on $Fol_{   \C_n^g \wr\wr
D_F}^{\C^g\wr \wr D_F} $. Let $\delta $  a
imaginary point and consider the following graphs:
\begin{eqnarray} \label{W_n} W_n=  
\overline{\C_n^g \wr D_F}  ,
\end{eqnarray}
and
\begin{eqnarray} \label{W_n'}  W_n'=     
\overline{\C_n^g \wr \wr D_F }
,\end{eqnarray}
defined by :
$$  V(\overline{\C_n^g  \wr D_F
})=V(\overline{\C_n^g \wr \wr D_F })= V( \C_n^g \wr  D_F ) \cup \{\delta\}$$
and set of edges are given by 
$$ E(W_n)=E(    \C_n^g \wr D_F)\cup
\{ (x,\delta);\ x\in  V    (\C_n^g \wr D_F) \ and \ \exists y\in V(W)\  \ (x,y)\in E(W) \}$$
and
$$ E(W_n')=E(   C_n^g \wr\wr D_F)      )\cup
\{ (x,\delta);\ x\in V(    C_n^g \wr\wr D_F ) \ and \ \exists y\in V(W) \ \ (x,y)\in E(W') \}.$$
 \\
 Let respectively $d$ and  $d'$  be the distances on  $W$
 and  $W'$, given by edges of these graphs.
 $W_n$ are  $W_n'$ are rough  isometric with  constants
 independant of $n$. With the  notations  of 
 definition 3.7 in  \cite{woess},
 we have  $A=3$  and $B=0.$ \\
 Indeed, consider  $$id: (V(W_n),d)\rightarrow (V(W'_n),d).$$ 
For all  $x,y\in V(W_n)=V(W_n'),$ we have: $$\frac{1}{3}d(x,y)\leq d'(x,y)\leq 3d(x,y).$$
 Thus  the respective Dirichlet forms  $\Erond$ and 
 $\Erond'$ for simple random walks on $W_n$ and 
 $W_n' $  satisfy: there 
  exists $c_1,c_2>0$ such that for all  $f: V(W_n )
  \rightarrow \R$ we have,
 $$ c_1 \Erond(f,f) \leq \Erond'(f,f) \leq c_2\Erond(f,f), $$
 with 
$$\Erond (f)= \underset{(x,y)\in E(W_n) }{\sum} (f(x)-f(y))^2, $$
and
 $$\Erond' (f)= \underset{(x,y)\in E(W_n') }{\sum} (f(x)-f(y))^2 .$$
Now, let $U \subset V(      C_n^g \wr\wr D_F  )$
and take  $f=1_U$, we get : 
 $$  c_1 |\partial_W U| \leq |\partial_{W'} U | \leq c_2 |\partial_W U|.$$
  \end{enumerate}
Hence,  we have proved that  (\ref{Fol^W}) carry to
   $Fol_{   \C_n^g \wr\wr
D_F}^{\C^g\wr \wr D_F} $, so we deduce:
 \begin{prop} \label{usefol}
For all  $\gamma>0$, there  exists $\beta>0$ such
that for all $c>0, \ \Q$ a.s on  the set 
$|\C|=+\infty$ and for large enough  $n$, we have : 
 \begin{eqnarray} \label{Fol^W'}
Fol_ {         C_n^g \wr\wr D_F}^{   C^g \wr\wr D_F} (k) \succeq \begin{cases}
 \; F(k)^k & \text{if $k < c n^{\gamma},$} \\
  \; (F( k)^ {\beta k^d} & \text{if  $k\geq   c n^{\gamma}.$}
\end{cases}
\end{eqnarray}
\end{prop}

Now we are able to get an upper bound of 
$\tilde{\P}^{\omega}_o( Z_{2n}=o)$ and then  an upper bound   of our
functional.
Let  $\tau_n = \inf\{ s \geq 0 \ ; \ Z_s \not\in V(       C_n^g \wr\wr D_F)    \}$ .\\
We have,
$$\tilde{\P}^{\omega}_o(Z_{2n}=o) =\tilde{\P}^{\omega}_o(Z_{2n}=o  \
\text{and} \ 
\tau_n \leq n) + \tilde{\P}^{\omega}_o(Z_{2n}=o  \ \text{and} \ 
\tau_n > n   ).$$
The first term is zero since the walk can not  go
out the box   $V(        C_n^g \wr\wr D_F)  $ 
before time  $n$. \\
  The second term can be bounded with
  proposition \ref{usefol}.  
Let :
   \begin{eqnarray} 
\H(k)= \begin{cases}
 \; F(k)^k & \text{if $k < c n^{\gamma},$} \\
  \; (F( k)^ {\beta k^d} & \text{if  $k\geq   c n^{\gamma}.$}
\end{cases}
\end{eqnarray}
 - $\H$ is increasing and we can define an inverse
 function by  
 $$\H^{-1}(y)=\inf \{x;\H(x) \geq y \}. $$
 - Besides, with the help of   (\ref{Fol^W'}),  $$ Fol_{     C_n^g \wr\wr
 D_F}^{   C^g \wr\wr D_F}  \succeq
 \H. $$
 - $\C$ and  $D_F$ have bounded valency and from 
 formula of  $\tilde{m}$
 and $\tilde{a}$ (see  (\ref{atilde}) and
 (\ref{mtilde})) we have :
 $\underset{ V(W') }{\inf} \tilde{m}\geq \frac{1}{2d} >0$ et 
 $\underset{ E(W') } {inf} \tilde{a} >0 .$ 
  \\
 Thus, (see  theorem 14.3 in\cite{woess} for
 example) there  exists constants  $c_1,c_2 \ and  \
 c_3>0$ such that   $$\tilde{\P}^{\omega}_o(Z_{2n}=o  \ \text{et} \ 
\tau_n > n   )\preceq   u(n)$$
where  $u$ is  solution of the differential equation :
$$\left\lbrace
\begin{array}{l}
 u'=-\frac{u}{c_2(\H^{-1} (c_3/u))^2} , \\
 u(0)=c_1 .\\
\end{array}
\right.$$
Replacing   $F(k)$ by  $e^{k^{\beta} }$ into  $\H$, we get
the expression of  $\H^{-1}$:  
\begin{eqnarray}  \label{F^{-1}}
\H^{-1}(y)= 
\begin{cases}
 \;c(ln(y))^{\frac{1-\alpha}{ 1+\alpha}}  &
 \text{if $  1 \leq y <  e^{c n^{\frac{\gamma (1+\alpha)}{ 1-\alpha}} }   ,$} \\
  \;    cn^{\gamma}  &\text {if $ e^{c n^{\frac{\gamma (1+\alpha)}{ 1-\alpha}} }  \leq y <   
  e^{c n^{\frac{\gamma (d+\alpha(2-d) ) }{ 1-\alpha}} }   ,$} \\
  \;   c(ln(y))^{  \frac{1-\alpha}{d+\alpha(2-d)}} & 
  \text{if  $  
  e^{
  c n^
  {
  \frac{\gamma (d+\alpha(2-d) ) }
  { 1-\alpha}
  } 
  }    
     \leq y .$}
\end{cases}
\end{eqnarray}
Resolving the differential equation in the
different cases, we get :  
\begin{eqnarray*}  
u(t)= 
\begin{cases}
 \; c e^{-ct^{\frac{ 1+\alpha}{3-\alpha} }}  &
 \text{if $  t \leq  cn^{\frac{ \gamma(3-\alpha)}{ 1-\alpha} }   ,$} \\
  \; c  e^{cn^{\gamma\frac{1+\alpha}{1-\alpha}}}
  e^{-ct/n^{2\gamma} }  &\text {if $ cn^{\frac{ \gamma(3-\alpha)}{ 1-\alpha}} < t \leq
  cn^{\frac{ \gamma(d+2-d\alpha)}{ 1-\alpha} } + n^{\frac{\gamma(3-\alpha)}{1-\alpha}}    ,$} \\
  \;  ce^{-(ct- c'n^{\frac{ \gamma(d+2-d\alpha)}{ 1-\alpha} } -cn^{\frac{\gamma(3-\alpha)}{1-\alpha}}  
  )^{ \frac{d+\alpha(2-d) }{ 2+d-d\alpha}} } 
    & 
  \text{if  $  cn^{\frac{ \gamma(d+2-d\alpha)}{ 1-\alpha} }  + n^{\frac{\gamma(3-\alpha)}{1-\alpha}}    \leq t . $}
\end{cases}
\end{eqnarray*}\\
(Each $c$ design a different constant.)\\
\\
Now whe choose   $\gamma $  such that   $0 < \gamma < \min (\frac{1-\alpha}{d+2-d\alpha },
\frac{1-\alpha}{3-\alpha})$, then we get :\\

there  exists $c=c(p,d,\alpha,\lambda)>0$ such that
$$ u(2n) \leq e^{-cn^{\eta}},$$ 
with  $\eta= \frac{ d+\alpha (2-d) }{ 2+d(1-\alpha) }.$\\
So, $\Q$ a.s on the set $|\C|=+\infty$,  and for
large enough  $n$ (which depends on the cluster  
$\omega$),
$$\tilde{\P}^{\omega}_o(Z_{2n}=o    ) \preceq  e^{-n^{\eta}}.  $$

By proposition  \ref{fait0},  we deduce that   $\Q$
a.s 
on the event  $|\C|=+\infty$ and for large enough  $n$,
\begin{eqnarray} \label{prel}\mathbb{E}_0^{\omega} 
			 ( \underset {x; L_{x;2n}>0 }{\prod}
			 \P_{d_0 }^{D_F}
			 ( Y^{D_F}_{L_{x;2n} }  =d_0) 
			 \ \        
1_{       \{  X_{2n}=0 \}        }         )\preceq  e^{-n^{\eta}}.
\end{eqnarray} 
By our choice of graph $D_F$,  there exists $C_1,
C_2>0$  such that for all  $n\geq 1$:
\begin{eqnarray} \P_{d_0}^{D_F}
			 ( Y^{D_F}_{n }  =d_0 )  &\geq& C_1 e^{-
			 (C_2n)^{\alpha} },\\
			 &\geq& e^{- \lambda_0 n^{\alpha}} \label{C},
						 \end{eqnarray}
for some   $\lambda_0>0$.\\
\\
From  (\ref{prel}) and  (\ref{C}), we get that
there exists $\lambda_0>0$ such that $\Q$ a.s on
the set $\C|=+\infty$ and for large enough  $n$,
\begin{eqnarray} \label{preli}
\mathbb{E}_0^{\omega} 
( e^{- \lambda_0 \underset {x; L_{x;2n}>0
}{\sum} L_{x;2n}^{\alpha}  }  1_{\{ X_{2n}=0\}})
\preceq  e^{-n^{\eta}}.
\end{eqnarray}
To  conclude, it remains only to prove that we can
suppress the indicatrice function and  that we can
extend the inequality  (\ref{preli}) to all  $\lambda>0$. 
We explain this in 3 steps.
\begin{enumerate} [1.]
\item  First of all, notice that is is sufficient
to prove (\ref{ext2})  only for one value of
$\lambda$. Indeed, let  $\lambda >0$, assume that
for  $\lambda=\lambda_0$, we have:
\begin{eqnarray}\label{vpreli}
 \mathbb{E}_0^{\omega} 
( e^{- \lambda_0 \underset {x; L_{x;n}>0
}{\sum} L_{x;n}^{\alpha}  }  )
\preceq  e^{-n^{\eta}}.
\end{eqnarray}
-If  $\lambda\geq \lambda_0$, (\ref{vpreli}) is
true because we can replace  $\lambda_0$ by 
$\lambda$ using merely the  decrease.\\
-If  $\lambda < \lambda_0 $, we write 
\begin{eqnarray*} 
\E_0^{\omega} [  e^{-\lambda \underset{
x;L_{x;n} >0}{\sum}  L_{x;n}^{\alpha} }
  ]
  &=& 
  \E_0^{\omega} [  (e^{-\lambda_0 \underset{
  x;L_{x;n} >0}{\sum}  L_{x;n}^{\alpha} }
   )^{\frac{\lambda}{\lambda_0}}] \\
  &\leq & 
  (\E_0^{\omega} [  e^{-\lambda_0 \underset{
  x;L_{x;n} >0}{\sum}  L_{x;n}^{\alpha} }
  ]
   )^{\frac{\lambda}{\lambda_0}} \\
  &\ & \text{(Jensen inequality applied to concave
  function}\  x\rightarrow 
  x^{   \frac{\lambda}{\lambda_0} }.)\\
  &\preceq & e^{-n^{\eta}} .
\end{eqnarray*}

\item To take out the indicatrice function, we use
the following lemma: 
\begin{lem}  \label{bidule} For all  $m\geq 0$, we
have :
$$\P_0^{\omega} (\sum_x L_{x;n}^{\alpha}=m )^2
\leq 2d(2m+1)^d \  \P_0^{\omega} (    \sum_x
L_{x;2n}^{\alpha}\leq 2m  \text{ et } X_{2n}=0
).$$
\end{lem}

\begin{proof}
\begin{eqnarray*}
[\P_0^{\omega}(  \sum_x L_{x;n}^{\alpha} =m)]^2 &=&   \Bigl(   \underset{ h\in B_m(\C)}{\sum}    
\P_0^{\omega} (  \sum_x L_{x;n}^{\alpha}  =m ; X_n=h)
\Bigr) ^2  \\
&=& 
   \Bigl(   \underset{ h\in B_m(\C)}{\sum}    
 \sqrt{\nu(h)}  \times 1/\sqrt{\nu(h)}  \times \\
  &\ & \hspace{4.4cm} \P_0^{\omega}(    \sum_x L_{x;n}^{\alpha}=m ; X_n=h)
\Bigr) ^2    \\
&\leq&  \nu(B_m(\C)) \underset{ h\in B_m(\C) }{\sum} (1/\nu(h))   \P_0^{\omega}
(\sum_x L_{x;n}^{\alpha}=m ; X_n=h)^2\\
& \ &\ \hspace{0cm} 
\text{(Cauchy-Schwarz inequality) } \\
&\leq& 2d (2m+1)^d   
 \underset{ h\in B_m(\C) }{\sum}   
 \P_0^{\omega}( \sum_x L_{x;n}^{\alpha}=m ; X_n=h) \times \\
 &\ &  \hspace{3.5cm} \P_h^{\omega}(     \sum_x L_{x;n}^{\alpha}=m ; X_n=0)(1/\nu(0)) \\
& \ &\ \hspace{0cm}  
\text{ (by reversibility ) } \\
&\leq& 2d (2m+1)^d   
 \underset{ h\in B_m(\C) }{\sum}    
 \P_0^{\omega}(   \sum_x L_{x;n}^{\alpha}=m ; X_n=h) \times \\
&\ & \hspace{3.3cm}  \P_0^{\omega}(   \sum_x
L_{x;[n;2n]}^{\alpha}=m ; X_n=h ; X_{2n}=0)\\
 & \ &\     \hspace{0cm}    ( \text{where  }    
 L_{x;[n;2n]}= \#\{ i\in[n;2n];\  X_i=x  \} )\\
  &\leq& 2d (2m+1)^d \ \P_0^{\omega}(    \sum_x
  L_{x;2n}^{\alpha}  \leq 2m;X_{2n}=0). 
\end{eqnarray*}
because  $\{ \sum_x L_{x;n}^{\alpha}   =m  \text{ et
}  \sum_x L_{x;[n;2n] }^{\alpha} =m\}\subset \{   \sum_x L_{x;2n}^{\alpha}
\leq 2m\},$
since for  $\alpha \in [0,1[,$ we have :   $$L_{x;2n}^{\alpha}\leq
(  L_{x;n} +
L_{x;[n;2n]})^{\alpha} \leq   L_{x;n} ^{\alpha} +
L_{x;[n;2n]}^{\alpha} . $$
\end{proof}
Then we write,
\begin{eqnarray*}
\E_0^{\omega}(    e^{-\lambda_0  \sum_x L_{x;2n }^{\alpha}       } 1_{\{ X_{2n}=0   \}}  ) 
&=& \underset{  m \geq 1}{\sum}   e^{-\lambda_0
m} \ 
\P_0^{\omega}(    \sum_x L_{x;2n }^{\alpha}  =m          ; X_{2n}=0 )\\
&=&  (1-e^{-\lambda_0}) \underset{  m \geq 1}{\sum}
  e^{-\lambda_0 m} \
\P_0^{\omega}(     \sum_x L_{x;2n }^{\alpha} \leq m ;
X_{2n}=0 ),\\
\end{eqnarray*}
since  $  \{   \sum_x L_{x;2n }^{\alpha}=m\}=
\{  \sum_x L_{x;2n }^{\alpha} \leq m\}-\{    
\sum_x L_{x;2n }^{\alpha} \leq m-1\}.$ Thus, we
have,
\begin{eqnarray*}
 \E_0^{\omega}(   e^{-\lambda_0}   \sum_x
 L_{x;2n }^{\alpha} 1_{\{ X_{2n}=0   \}}  ) 
 & \geq & 
  (  1- 
  e^{   - \lambda_0 }   
   )  
 \underset{m \geq 1}{\sum}  
 e^{-2\lambda_0 m}  \P_0^{\omega}(     \sum_x L_{x;2n }^{\alpha} \leq 2m ; X_{2n}=0 )\\
&\ & \hspace*{0cm}(\text{we add only the even $m$}   \  )\\
&\geq& (  1- 
  e^{   - \lambda_0 }   
   ) 
    \underset{  m \geq 1}{\sum} \frac{1}{2d
(2m+1)^d}e^{-2\lambda_0 m} \ 
[\P_0^{\omega}(     \sum_x L_{x;n }^{\alpha}=m)]^2 \\
&\ & (\text{by lemma  } \ref{bidule} ) \\
&\geq&      \underset{  m \geq 1}{\sum} e^{-\lambda_1 m}  [\P_0^{\omega}
(\sum_x L_{x;n }^{\alpha} =m)]^2 \\
&\ & (\text{for some  
$\lambda_1 > 2\lambda_0 $ }) \\
&\geq&    \Bigl( \underset{  m \geq 1}{\sum}
e^{-\lambda_1 m}
\Bigr)^{-1} \times \\
& \ & \hspace{2.8cm}\Bigl(
\underset{  m \geq 1}{\sum} e^{-m\lambda_1
}\  \P_0^{\omega}(\sum_x L_{x;n }^{\alpha}=m) 
\Bigr)\\
&\ & \hspace*{0cm} (\text{By Cauchy-Schwarz
inequality})\\
&\geq&   c_0\ \E_0^{\omega} [e^{ -\lambda_1  \sum_x
L_{x;n }^{\alpha}}  ] .\\
\end{eqnarray*}

\item We can  now conclude. By previous inequality and
by   (\ref{preli}), there  exists $\lambda_1 $
such that :  
$$   \E_0^{\omega} [e^{ -\lambda_1  \sum_x
L_{x;n }^{\alpha}}  ]  \preceq e^{-n^{\eta}}.$$
Then by step 1, we can extend this inequality to
all $\lambda_1$ . 
\end{enumerate}
Finaly we have proved :
\begin{prop} \label{1a}
 $Q \ a.s$ on  $|\C|=+\infty$ for large enough 
 $n$ and for all  $\lambda>0$ we have,
for all  $\alpha \in [0,1[,$
 $$\E_0^{\omega} (  e^{-\lambda \underset{
 x;L_{x;n} >0}{\sum}  L_{x;n}^{\alpha} })  \preceq 
 e^{-n^{\eta}},$$
 where  $\eta=\frac{ d+\alpha (2-d) }{ 2+d(1-\alpha) }.$
 \end{prop}

 \begin{rem} $\ $\\
 1) If  $\alpha=0$, we retrieve the Laplace 
 transform of the number of visited  points by the
 simple random walk on an infinite cluster. \\
 2)For  $\alpha =1$,  inequality is satisfied since
  $  \underset{ x;L_{x;n} >0}{\sum}  L_{x;n} =n$
  and   $\eta=1$ when $\alpha =1$.
  \end{rem}

{\bf{Lower bound}}\\
The proof falls into 4 steps.
\begin{enumerate} [1.]
\item By concavity of the function   $ \  x\mapsto x^{\alpha} $ 
for  $\alpha \in [0,1],$ we have :

\begin{eqnarray*}
 \underset{ x;L_{x;n} >0}{\sum}  L_{x;n}^{\alpha}  &\leq &
   N_n  (\underset{ x;L_{x;n} >0}{\sum}  \frac{ L_{x;n} }{N_n} )^{\alpha}  \\
                           &=&	 N_n^{1-\alpha} n^{\alpha} .    
\end{eqnarray*}

So,
\begin{eqnarray*}
  \E_0 (  e^{-\lambda \underset{ x;L_{x;n} >0}{\sum}  L_{x;n}^{\alpha} } 
  )  &\geq &
  \E_0( e^{-\lambda n^{\alpha} N_n^{1-\alpha}  } 
  )  \\
                           &\geq&	  \P_0 (\underset{ 0\leq i \leq n}
			   {\sup} D(0,X_i) \leq m 
			    ) 
			    e^{-\lambda V(m)^{1-\alpha} n^{\alpha}} .    
\end{eqnarray*}

where $V(m)=|B_m(\C)|$ stands for the  volume of the ball of   $\C$
centred at the origin with radius   $m$. \\
\item

By proposition  5.2 of  \cite{KL}, we have :
\begin{eqnarray} \label{deb}
\P_0( \underset{ 0\leq i \leq n}{\sup} D(0,X_i) \leq m 
  )\geq 
  e^{ -c(m+ \frac{n}{m^2}) }
\end{eqnarray}

\item By lemma 5.3 of \cite{KL}, there   exists $c>0$ such that  $Q$ a.s 
on  $|\C|=+\infty $ and for large enough  $n$, 
$$V(m)\geq cm^d.$$
\item So, we deduce, there exists $C>0$ such that  $Q$ a.s on
$|\C|=+\infty $ and for large enough  $n$,
$$\E_0 (  e^{-\lambda \underset{ x;L_{x;n} >0}{\sum}  L_{x;n}^{\alpha} } 
 )  \geq 
e^{-C(m+\frac{n}{m^2}  +\lambda n^\alpha m^{d(1-\alpha)} ) }$$
Taking  $m =n^{\frac{ 1-\alpha}{ 2+d(1-\alpha) } } $, we get :
  $$\E_0 (  e^{-\lambda \underset{x;L_{x;n} >0}{\sum}  L_{x;n}^{\alpha} } 1_{\{ X_n=0\}} )  \geq 
e^{-cn^{\eta}}, $$
with $\eta= \frac{ d+ \alpha(2-d)}{ d(1-\alpha) +2}$  and for all  
 $\alpha \in [0,1].$
 \end{enumerate}
Hence, we have proved :
\begin{prop} \label{1b}For all  $\alpha \in [0,1]$, $Q$ a.s  on  $|\C|=+\infty $
and for  large enough $n$,
$$
\E_0 (  e^{-\lambda \underset{x;L_{x;n} >0}{\sum}  L_{x;n}^{\alpha} }
)  \succeq 
e^{-cn^{\eta}},$$
with $\eta= \frac{ d+ \alpha(2-d)}{ d(1-\alpha) +2}. $
\end{prop}
Thus, the first assertion  of Theorem  \ref{ext} comes from proposition \ref{1a}
and proposition \ref{1b}.

\subsubsection{ Study of 
$ \E_0^{\omega}( 
\prod   L_{z;n}^{-\alpha}  ) $
}
 $\ $  \\
 We  assume $\alpha>1/2.$
 \\ \\
 {\bf{Upper bound}} \\
For this functional, one can take for all  $x\in \C$, $B_x=\L^1=(\Z,E_1
)$ ( if we take some $\L^r$, we get the same bound). We have :
$$ Fol_{\L^1}(k)= 2k.$$
We still use a random walk $Y$ which jumps  can be  represented by: \\
$$\includegraphics*[width=6cm]{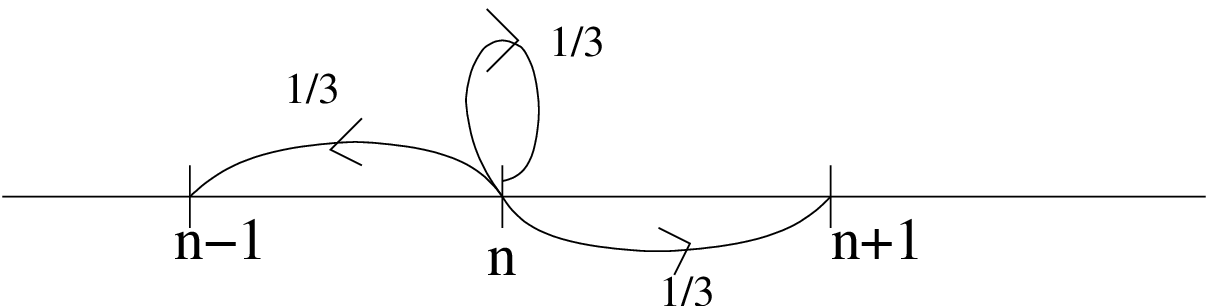}$$
Let   $\P^{\L^1 }$ be the  law of random walk $Y$.\\
 
As before, let  $W,W',W_n$ and  $W_n'$ be the  graphs defined  respectively by  
 (\ref{W}) (\ref{W'})
  (\ref{W_n}) and 
 (\ref{W_n'})  with    $D_F=\L^1=(\Z,E_1)  $.
With the help of  proposition 1.4 in \cite{KL}  and general properties 
of isoperimetry on wreath product, we deduce:\\
for all  $\gamma>0$, there exists $c,\beta >0$ such that $Q$ a.s on  
$|\C|=+\infty$ we have :
 \begin{eqnarray} \label{Fol^Wbis}
Fol_{  W_n  }^{W } (k) \succeq \begin{cases}
 \; k^k & \text{if $k < c n^{\gamma},$} \\
  \;  k^ {\beta k^d} & \text{if  $k\geq   c n^{\gamma}.$}
\end{cases}
\end{eqnarray}

With the same argument as in the upper bound of   section \ref{f1},  we carry 
(\ref{Fol^Wbis}) to  $Fol_{W_n' }^{    W' } $ 
by rough isometry between  graphs $W_n$ and  $W'_n$.  We get:\\
for all  $\gamma>0$, there  exists $\beta >0$ such that for all $c>0,\ Q$
a.s on  
$|\C|=+\infty$  we have:
 \begin{eqnarray} \label{Fol^W'bis}
Fol_{ W_n'}^{   W'} (k) \succeq \begin{cases}
 \; k^k & \text{if $k < c n^{\gamma},$} \\
  \;  k^ {\beta k^d} & \text{if  $k\geq   c n^{\gamma}.$}
\end{cases}
\end{eqnarray}

In order to get an upper bound of  $\tilde{\P}^{\omega}_o(Z_{2n}=o) $,
let again:
$$\tau_n=\inf\{ s\geq 0;\ Z_s\not\in  V(     W_n') \}.$$
We still have   $\tilde{\P}^{\omega}_o(Z_{2n}=o) =
\tilde{\P}^{\omega}_o(Z_{2n}=o \  \text{and } \ \tau_n>n ) $. We use the
same  way to get the upper bound from (\ref{Fol^W'bis}).\\
 Inequality  (\ref{Fol^W'bis}) implies:
  \begin{eqnarray}  \label{Fol^W'bibis }
\forall k\geq 0 \ \ \  Fol_{     W_n'}^{W'}(k) \succeq \J_N (k)= 
\begin{cases}
 \; 1 & \text{if $k < c n^{\gamma}$,} \\
  \;  N^ {\beta d' k^d} & \text{if  $k\geq   c n^{\gamma},$}
\end{cases}
\end{eqnarray}\\
where $N \leq  cn^{\gamma}$.\\
$\J_N$ is increasing and we can compute  $\J_N^{-1}$:
 \begin{eqnarray*} \J_N^{-1} &=& \inf\{x; \ \J_N(x)\geq y \} \\
 &=& \begin{cases}
 \; cn^{\gamma} & \text{if $1\leq y < N^{cn^{d\gamma}} ,$}\\
 \; c\Bigl(\frac{\ln(y)}{\ln(N)} \Bigr)^{1/d} &\text{if $ N^{cn^{d\gamma}}  \leq y.$}
 \end{cases}
 \end{eqnarray*}
 
 \begin{rem} Let   \begin{eqnarray} 
\J(k)= \begin{cases}
 \; k^k & \text{if $k < c n^{\gamma},$} \\
  \;  k^ {\beta k^d} & \text{if  $k\geq   c n^{\gamma}.$}
\end{cases}
\end{eqnarray}
 Inequality   (\ref{Fol^W'bis})  can be read   $Fol_{   W_n'}^{W'}(k) \succeq \J(k)$.
 $\J$ is increasing but the form of $\J$  does not enable us to compute
 an inverse and for  this reason we use $\J_N$ for the lower bound of
  $Fol_{  W_n'}^{ W'}(k) $ instead of $\J$. 
   \end{rem}
 
 $\C $ and $\L^1$ have bounded valency so we still have 
 $\underset{ V(W') }{\inf} \tilde{m}\geq 2 >0$ and  
 $\underset{ E(W' )  }{inf} \tilde{a} >0 .$ Thus with the same tools
as in section \ref{f1} we get, there  exists  constants $c_1,c_2 \ and \
c_3>0$  such that 
  $$\tilde{\P}^{\omega}_o(Z_{2n}=o  \ \text{and} \ 
\tau_n > n   )\preceq   u(n)$$
where $u$ is  solution of the differential equation:
$$\left\lbrace
\begin{array}{l}
 u'=-\frac{u}{c_2(\J_N^{-1} (c_3/u))^2} , \\
 u(0)=1/2 .\\
\end{array}
\right.$$
Solving this equation, we obtain:
\begin{eqnarray*}  
u(t)= 
\begin{cases}
 \;  e^{-ct/n^{2\gamma}  }& \text{if } 
   t \leq  t_0:=  cn^{\gamma(d+2) } \ln(N), \\
  \;    e^{-(\    c(\ln(N)^{2/d}(t-t_0) + \ln(1/u(t_0))^{\frac{d+2}{d} } 
         \   )^{\frac{d}{d+2}}   }  &\text {if $ t>t_0 .$} \\
\end{cases}
\end{eqnarray*}\\
Chosing  $\gamma < \frac{1}{d+2}$ and taking 
$N=cn^{\gamma}$, we obtain in $t=n$:\\
 $Q$ a.s   on the event  $|\C|=+\infty$ and for large enough  $n$,
 $$\tilde{\P}^{\omega}_o(Z_{2n}=o) \preceq
  e^{-n^{\frac{d}{d+2} } \ln(n)^{\frac{2}{d+2} } }.$$
So with proposition \ref{fait0}, we deduce :
  \begin{prop}
  \label{avecproba} There exists a constant $C>0$ such that   $Q$ a.s on 
  $|\C|=+\infty$ and for large enough  $n$,
   \begin{eqnarray} \label{oufrate} \mathbb{E}_0^{\omega} 
   ( \underset {x; L_{x;2n}>0 }{\prod} \P_{0}^{\L^1}
	( Y_{L_{x;2n} }  =0) 
	\ \        
1_{       \{  X_{2n}=0 \}        }         )  
 \leq
  e^{- C n^{\frac{d}{d+2} } \ln(n)^{\frac{2}{d+2} } } .
  \end{eqnarray}
  \end{prop}
For the walk  $Y,$ we know that there  exists  $c_0>0$ such that  
$ \P^{\L^1}_0(Y_n=0) \sim \frac{c_0}{n^{1/2}}.$ In particular,  
\begin{eqnarray} \label{cellela} \exists c_1 >0 \ \forall n\geq 1 \ \  
\P^{\L^1}_0(Y_n=0) 
\geq \frac{c_1}{n^{\alpha}},
\end{eqnarray}
with $c_1\leq 1.$
So, for $\alpha >1/2$ we can find  $A>0$ and  $c_2>0 $ such that 
\begin{eqnarray}  \label{naif} \forall n\geq 1 \ \  
\P^{\L^1}_0(Y_n=0) \geq
\begin{cases}
\;\frac{1}{n^{\alpha}} &\text{if} \ n\geq A,\\
\;\frac{c_2}{n^{\alpha}} &\text{if} \ n<A,
\end{cases}
\end{eqnarray}
with  $c_2\leq 1$.  If we directly use the lower bound    (\ref{naif}) in  
(\ref{oufrate})    at time  $L_{x;2n}$, it appears a supplementary  factor 
 $c_2^{\#\{ x;\ 0< L_{x;2n}< A\}}$ on which we do not have control. \\
 So we put :
 $$N_{n,2}=\#\{x;\ L_{x;n}\geq 2\},$$
 which is the number of visited points at least twice by the walk $X$.
 And for  $\varepsilon_1, 
 \varepsilon_2>0,$ consider the following events :
 \begin{eqnarray*} 
 A_1&=& \{N_{2n}\leq \varepsilon_1\  n^{\frac{d}{d+2} } \ln(n)^{\frac{2}{d+2} } \},\\
A_2&=& \{N_{2n}\geq \varepsilon_1 \  n^{\frac{d}{d+2} }
\ln(n)^{\frac{2}{d+2} } \text{ and } 
 N_{2n,2}\geq  \varepsilon_2 \ n^{\frac{d}{d+2} } \ln(n)^{\frac{2}{d+2}
 }\},\\
   A_3&=& \{N_{2n}\geq \varepsilon_1\  n^{\frac{d}{d+2} }
 \ln(n)^{\frac{2}{d+2} } \text{ and } 
 N_{2n,2}\leq  \varepsilon_2 \ n^{\frac{d}{d+2} } \ln(n)^{\frac{2}{d+2}
 }\}.
 \end{eqnarray*}
We have 
 \begin{eqnarray} \label{add}
 \E_0^{\omega}(\prod_{x;\  L_{x;2n}>0}
 L_{x;2n}^{-\alpha} \  1_{\{X_{2n}=0\}} ) &=&
 \E_0^{\omega}(\prod_{x;\  L_{x;2n}>0}
 L_{x;2n}^{-\alpha} \  1_{\{X_{2n}=0\}} \ 1_{A_1} ) \\ \nonumber
  &+&
 \E_0^{\omega}(\prod_{x ;\ L_{x;2n}>0}
 L_{x;2n}^{-\alpha}  \ 1_{\{X_{2n}=0\}} \ 1_{A_2})   \\ \nonumber
&+& \E_0^{\omega}(\prod_{x;\  L_{x;2n}>0}
L_{x;2n}^{-\alpha} \  1_{\{X_{2n}=0\}} \  1_{A_3}).
 \end{eqnarray}
Let us examin these three 3 terms.
\begin{enumerate} [1.]
\item For the  term corresponding to  $A_1$, we write :
\begin{eqnarray*} 
  \E_0^{\omega}(\prod_{x;\  L_{x;2n}>0}
  L_{x;2n}^{-\alpha} \  1_{\{X_{2n}=0\}} \ 1_{A_1} )
 & = &  \E_0^{\omega}(\prod_{x;\  L_{x;2n}>0} 
 \frac{c_1}{L_{x;2n}^{\alpha}} \  \times 
 \prod_{x;\  L_{x;2n}>0}  \frac{1}{c_1} \times
 1_{\{X_{2n}=0\}} \ 1_{A_1} ) \\
 &\leq&
 \mathbb{E}_0^{\omega} 
   ( \underset {x; L_{x;2n}>0 }{\prod}
   \P_{0}^{\L^1}( Y_{L_{x;2n} }  =0)  \times 
   (\frac{1}{c_1})^{N_{2n}} \\
   &\ & \hspace{5.5cm} \times 1_{\{X_{2n}=0\}} \ 1_{A_1} ) \\
 &\ & \hspace{0.1cm}(\text{par } \ref{cellela} ) \\
 &\leq& 
 \mathbb{E}_0^{\omega} 
   ( \underset {x; L_{x;2n}>0 }{\prod}
   \P_{0}^{\L^1}( Y_{L_{x;2n} }  =0)  
    1_{\{X_{2n}=0\}}  )\\
    &\ & \hspace{4.8cm} \times   
    (\frac{1}{c_1})^{ \varepsilon_1\  n^{\frac{d}{d+2} } \ln(n)^{\frac{2}{d+2} }  } \\
 &\leq& e^{-(C +\varepsilon_1 \ln( c_1) ) n^{\frac{d}{d+2} } \ln(n)^{\frac{2}{d+2} } } . \\
 &\ & \hspace{0.1cm}(\text{by proposition  } \ref{avecproba} ) \\
 \end{eqnarray*}
Now,  choosing   $\varepsilon_1$ small enough   ( recall that $\ln(c_1)
\leq 0$),  we deduce that there exists a constant 
$C_1>0$ such that $Q$ a.s on  $|\C|=+\infty$, we have, 
\begin{eqnarray} \label{A_1borne}
\E_0^{\omega}(\prod_{x;\  L_{x;2n}>0}
L_{x;2n}^{-\alpha} \  1_{\{X_{2n}=0\}} \ 1_{A_1} )
\leq e^{- C_1 n^{\frac{d}{d+2} } \ln(n)^{\frac{2}{d+2} } } .
\end{eqnarray}
 \item For the second  term,  we notice that on the event $A_2$ the
 product $\prod_{x ;\ L_{x;2n}>0}
 L_{x;2n}^{-\alpha}$ is less than  $ (1/2)^{\varepsilon_2  n^{\frac{d}{d+2} }
 \ln(n)^{\frac{2}{d+2} } }$. Thus there exists a constant  $C_2>0$ such
 that,
   \begin{eqnarray} \label{A_2borne}
\E_0^{\omega}(\prod_{x;\  L_{x;2n}>0}
L_{x;2n}^{-\alpha} \  1_{\{X_{2n}=0\}} \ 1_{A_2} )
\leq e^{- C_2 n^{\frac{d}{d+2} } \ln(n)^{\frac{2}{d+2} } } .
\end{eqnarray}
 \item For the last  term, we use the following lemma:
 \begin{lem} 
 \label{lemmeana} There  exists  $\varepsilon'>0 $ such that for all  
  $\varepsilon>0$,  there  exists  a  constant $C_3>0$ such that, for all
   $ n,N \geq 0$,
\begin{eqnarray} \label{varepsineg} 
 \P_0^{\omega}( N_n \geq  \varepsilon N  \text{ et } N_{n,2} \leq \varepsilon' N ) \leq
  e^{- C_3 N } .
  \end{eqnarray}
  \end{lem}
 \begin{proof}
 $\ $ \\
 \begin{enumerate} [$\bullet$]
 \item Let  $\tau_0=0 $  and for  $k\geq 1$ let,
 $$ \tau_k=\min\{s\geq \tau_{k-1};\ X_s\not\in \{X_0,X_1,...,X_{s-1}\} \   \}.$$
 The  $\tau_k$ represent  instants when the walk $X$ visits a new  point. 
 Consider now, the variables  $\epsilon_k$ defined by :
 \begin{eqnarray} \label{epsilonk}
 \epsilon_k= \begin{cases}
 \; 1 & \text{ if } X_{\tau_k}=X_{\tau_k+2}, \\
\; 0 & \text{ otherwise.}
 \end{cases}
 \end{eqnarray}
These variables have the following interpretation,  $\epsilon_k$   is
equal to  $1$ only when the new visited point   $X_{\tau_k}$ is
immediatly re visited after a backward and foward. The  $\epsilon_k$  are
not independent but  their laws are all some Bernouilli  with different
parameters. Besides, these parameters have a same lower  bound
$\delta>0$,  since the graph  $\C^g$  has bounded valency. 
\item Consider  the following filtrations, 
$$\G_m=\sigma( X_j;\ 0\leq j\leq m),$$
$$\F_m=\sigma( X_j;\ 0\leq j\leq \tau_m).$$
 $\epsilon_k $  are  $\G_{2+\tau_k}$ measurable and so  $\F_{k+2}$
 measurable. For all  $\lambda>0$ and for all
$L>0$, 
we can write,
 \begin{eqnarray} \nonumber
 \E_0^{\omega} (e^{-\lambda \sum_{k=1}^L   \epsilon_k }) &=&
 \E_0^{\omega} (e^{-\lambda \sum_{k=1}^{L-2}   \epsilon_k }   \ 
 \E_0^{\omega} (e^{-\lambda   (\epsilon_{L-1}
 +\epsilon_L )} |\F_{L}) \ ) \\
 &\leq & \label{onvaitere}\E_0^{\omega} (e^{-\lambda \sum_{k=1}^{L-2}   \epsilon_k }   \ 
 \E_0^{\omega} (e^{-\lambda   \epsilon_L }
 |\F_{L}) \ ) .
 \end{eqnarray}
\item For the term  
$\E_0^{\omega} (e^{-\lambda    \epsilon_L } |\F_{L}), \ $ we have:
 \begin{eqnarray} \nonumber
 \E_0^{\omega} (e^{-\lambda    \epsilon_L } |\F_{L}) &=&
 e^{-\lambda} \P_0^{\omega} (\epsilon_L=1|\F_{L} ) + 
 \P_0^{\omega} (\epsilon_L=0|\F_L) \\ \label{leterme}
 &=& 1 + (e^{-\lambda} -1) \P_0^{\omega} (\epsilon_L=1|\F_{L} ).
 \end{eqnarray}
 
Now, we want a lower bound of  $\P_0^{\omega} (\epsilon_L=1|\F_{L} )$.
 We have successively:
  \begin{eqnarray} \nonumber
  \P_0^{\omega} (\epsilon_L=1|\F_{L} ) &=&
  \P_0^{\omega} (\epsilon_L=1|X_{\tau_L}  ) \\
  & \ & \text{( Markov property)} \nonumber \\
  &=& \nonumber \underset{x;    \  \P_0^{\omega} (X_{\tau_L}
  =x) >0              }{\sum} 1_{\{X_{\tau_L}
  =x\}} \ 
  \P_0^{\omega} (\epsilon_L=1| X_{\tau_L} =x) \\
  &\geq& \label{aaa}\delta ^2.
 \end{eqnarray}
Last inequality comes from the fact that the 
 graph $\C^g$ has bounded  valency, so in each point 
  $x$ the  probability to do a backward and foward is greater than 
  $\delta^2$ (with 
 $\delta \geq 1/2d$).
\item So, we deduce from  (\ref{leterme}) and  (\ref{aaa}) that,
 $$   \E_0^{\omega} (e^{-\lambda    \epsilon_L } |\F_{L}) \leq
 1 + (e^{-\lambda} -1)\delta^2. $$
Iterating  (\ref{onvaitere}), we get,
 \begin{eqnarray} \label {ineginte}
 \E_0^{\omega} (e^{-\lambda \sum_{k=1}^L   \epsilon_k })
 \leq  (1 + (e^{-\lambda} -1)\delta^2)^{ \lfloor
 L/2 \rfloor},
 \end{eqnarray}
where $\lfloor a \rfloor$ stands for the   whole number portion of $a$.
 Let: $$a_{\lambda} =-\ln( 1 + (e^{-\lambda} -1)\delta^2 )>0.$$  By 
 Bien-aymé inequality, we deduce,
 $$
 \P_0^{\omega} ( \sum_{k=1}^L  \epsilon_k \leq \varepsilon' L) 
  \leq
   e^{ \varepsilon'\lambda L -a_{\lambda} \lfloor
   L/2  \rfloor}.$$
Using  $\lfloor L/2 \rfloor$ for  $L\geq 2,\ L\leq 3$, we get :
   $$ \P_0^{\omega} ( \sum_{k=1}^L  \epsilon_k \leq \varepsilon' L) 
    \leq 
    e^{-\lfloor L /2\rfloor (a_{\lambda } -3
   \lambda \varepsilon') }.$$
Note that this last inequality is still valid for  $L=1$.\\
 Fix  $\lambda>0$, ( by  example $\lambda=1$) then we can choose
 $\varepsilon'$ small enough such that  $
 a_{\lambda} -3\varepsilon'>0$.
We deduce the existence  of  constant  $b$ 
 such that :
 \begin{eqnarray} \label{souhait}
 \P_0^{\omega} ( \sum_{k=1}^L   \epsilon_k \leq \varepsilon' L) 
 \leq  e^{-b L}.
 \end{eqnarray}
 \item Now, notice that 
 $$ \{ N_n \geq  \varepsilon N  \text{ et } N_{n,2} \leq \varepsilon' N \}
 \subset \{ \sum_{k=1}^{\varepsilon N}   \epsilon_k \leq \varepsilon' N  \}.$$
Indeed,  first if  $N_n\geq  \varepsilon N $ that means that at least 
  $  \varepsilon N $  new points have been visited. Secondly if  there
  are less than  $ \varepsilon' N$ points visited more than twice then
  there are less than  
   $ \varepsilon' N$  points which have  been immediatly visited after their
   first visit. Finaly we have: 
 $$\P_0^{\omega}(N_n \geq  \varepsilon N  \text{ and }
  N_{n,2} \leq \varepsilon' N ) \leq  e^{-\varepsilon b N}.$$
  \end{enumerate}

 \end{proof}
We can now get an upper bound of the term
corresponding to  $A_3$  The  product is less than
 $1$, so we can write: 
 \begin{eqnarray*}
 \E_0^{\omega}(\prod_{x;\  L_{x;2n}>0}
 L_{x;2n}^{-\alpha} \  1_{\{X_{2n}=0\}} \ 1_{A_3} )
\leq \P_0(A_3)  
\end{eqnarray*}
Let $\varepsilon_1$ small enough satisfying the
first point (event $A_1$ ),  lemma \ref{lemmeana} 
 with $\varepsilon=\varepsilon_1$  give us the
 existence of $\varepsilon'$ 
 such that  (\ref{varepsineg}). Then we take 
 $\varepsilon_2=\varepsilon'$ and we deduce there 
 exists a constant $C_3>0$  such that,
 \begin{eqnarray*}
 \P_0^{\omega}(A_3) \leq 
 e^{- C_3   n^{\frac{d}{d+2} } \ln(n)^{\frac{2}{d+2} }   } .
  \end{eqnarray*}
So,
 \begin{eqnarray} \label{A_3borne}
 \E_0^{\omega}(\prod_{x;\  L_{x;2n}>0}
 L_{x;2n}^{-\alpha} \  1_{\{X_{2n}=0\}} \ 1_{A_3} )
\leq 
 e^{- C_3   n^{\frac{d}{d+2} } \ln(n)^{\frac{2}{d+2} }   } .
 \end{eqnarray} 
 \end{enumerate}

 Finaly, we deduce  from (\ref{A_1borne})
 (\ref{A_2borne}) and  (\ref{A_3borne}), the
 following property.
\begin{prop} \label{propindic} 
$Q$ a.s on  $|\C|=+\infty$ and for large enough 
$n$, for all  $\alpha >1/2,$
\begin{eqnarray*}  \mathbb{E}_0^{\omega} 
   ( \underset {x; L_{x;2n}>0 }{\prod} 
	\frac{1}{  L_{x;2n}^{\alpha} }  
	\ \        
1_{       \{  X_{2n}=0 \}        }         )  
 \preceq
  e^{-n^{\frac{d}{d+2} } \ln(n)^{\frac{2}{d+2} } } .
  \end{eqnarray*} 
\end{prop}

To get the upper bound of the second point of
 Theorem  \ref{ext}, it remains to take
out  the indicatrice $1_{       \{  X_{2n}=0 \}}$.  We use the same  way as in
the section \ref{f1}. We prove:

\begin{lem}  \label{bidule'} For all  $m\geq 0$, we have:
$$\P_0^{\omega} (\sum_x \ln( L_{x;n}) =m )^2
\leq 2d(2m+1)^d \  \P_0^{\omega} (    \sum_x
\ln(L_{x;2n} )\leq 2m  \text{ et } X_{2n}=0
).$$
\end{lem}
The proof is similar to lemma  \ref{bidule}.
We use in  particular :
$$  \ln(L_{x;2n}) \leq
\ln(L_{x;n} +L_{x;[n;2n]}) \leq \ln(L_{x;n})
+\ln(L_{x;[n;2n]}).$$
Then
\begin{eqnarray*}
\E_0^{\omega}(     \prod_x 
L_{x;2n }^{-\alpha}        1_{\{ X_{2n}=0   \}}  ) &=& 
\E_0^{\omega}(    e^{-\alpha \sum_x \ln(L_{x;2n
})        } 1_{\{ X_{2n}=0   \}}  ) 
\\
&=&
\underset{  m \geq 1}{\sum}   e^{-\alpha m} \ 
\P_0^{\omega}(    \sum_x \ln(L_{x;2n })   =m          ; X_{2n}=0 )\\
&=&  (1-e^{-\alpha }) \underset{  m \geq 1}{\sum}
  e^{-\alpha m} \
\P_0^{\omega}(     \sum_x \ln(L_{x;2n })  \leq m ;
X_{2n}=0 ).\\
 & \geq & 
  (  1- 
  e^{   - \alpha  }   
   )  
 \underset{m \geq 1}{\sum}  
 e^{-2\alpha m}  \P_0^{\omega}(     \sum_x
 \ln(L_{x;2n })  \leq 2m ; X_{2n}=0 )\\
&\geq& (  1- 
  e^{   - \alpha}   
   ) 
    \underset{  m \geq 1}{\sum} \frac{1}{2d
(2m+1)^d}e^{-2\alpha m} \ 
[\P_0^{\omega}(     \sum_x \ln( L_{x;n }) =m)]^2 \\
&\ & (\text{by  lemma } \ref{bidule'} ) \\
&\geq&      \underset{  m \geq 1}{\sum}
e^{-\alpha_1  m}  [\P_0^{\omega}
(\sum_x \ln(L_{x;n })  =m)]^2 \\
&\ & (\text{for some $\alpha_1 > 2\alpha $ }) \\
&\geq&   \Bigl( \underset{  m \geq 1}{\sum}
e^{- \alpha_1 m}
\Bigr)^{-1} \times \\
& \ & \hspace{2.8cm}\Bigl(
\underset{  m \geq 1}{\sum} e^{-\alpha_1 m
}\  \P_0^{\omega}(\sum_x \ln(L_{x;n }) =m) 
\Bigr)\\
&\ & \hspace*{0cm} (\text{by Cauchy-Schwarz
inequality})\\
&\geq&  c\ \E_0^{\omega} [e^{ -\alpha_1  \sum_x
\ln(L_{x;n } ) }  ] \\
&=&c
\E_0^{\omega}(     \prod_x 
L_{x;n }^{-\alpha_1}             )
.\end{eqnarray*}

So, with this last inequality and with proposition  
\ref{propindic},  we obtain the expected upper
bound for some value $\alpha_1$:
\begin{eqnarray} \label{preli'} \E_0^{\omega}(     \prod_x 
L_{x;n }^{-\alpha_1}             ) \preceq 
e^{-n^{\frac{d}{d+2} } \ln(n)^{\frac{2}{d+2} }
} .\end{eqnarray}
From this inequality at point  $\alpha_1$, we
extend this relation for all $\alpha >1/2$.
Let  $\alpha >1/2$. \\
 -If $\alpha \geq \alpha_1 $,  we can replace in (\ref{preli'})
  $\alpha_1$ by $\alpha$, by monotony	 in
  $\alpha$.\\
-If $\alpha < \alpha_1 $, we write 
\begin{eqnarray*} 
\E_0^{\omega} [  e^{-\alpha  \underset{
x;L_{x;n} >0}{\sum}  \ln(L_{x;n})    }
  ]
  &=& 
  \E_0^{\omega} [  (e^{-\alpha_1 \underset{
  x;L_{x;n} >0}{\sum}  \ln(L_{x;n})   }
   )^{\frac{\alpha}{\alpha_1 }}] \\
  &\leq & 
  (\E_0^{\omega} [  e^{-\alpha_1  \underset{
  x;L_{x;n} >0}{\sum}  \ln(L_{x;n})  }
  ]
   )^{\frac{\alpha }{\alpha_1 }} \\
  &\ & 
  \text{ ( Jensen inequality to
  concave function }
   \  x \rightarrow 
  x^{   \frac{\alpha}{\alpha_1}  } .)\\
  &\preceq & e^{-n^{\eta}} .
\end{eqnarray*}
So we have proved :
\begin{prop}  \label{f21}
$Q$ a.s  on the set  $|\C|=+\infty$ and for large
enough  $n$, for all $\alpha >1/2,$
\begin{eqnarray*}  \mathbb{E}_0^{\omega} 
   ( \underset {x; L_{x;n}>0 }{\prod} 
	\frac{1}{  L_{x;n}^{\alpha} }  
	\ \        
        )  
 \preceq
  e^{-n^{\frac{d}{d+2} } \ln(n)^{\frac{2}{d+2} } } .
  \end{eqnarray*} 
\end{prop}
{\bf{Lower bound }}\\

By concavity of the function $\ln$, we get: 
\begin{eqnarray*}
 \underset{ z;L_{z;n} >0}{\prod }  L_{z;n}^{-\alpha}  &= &
 e^{-\alpha  N_n  \underset{ z;L_{z;n} >0}{\sum }
 \frac{1}{N_n}ln( L_{z;n})} \\
 &\geq&  
   e^{-\alpha  N_n  ln( \underset{ z;L_{z;n} >0}{\sum } \frac{
   L_{z;n} } {N_n})} \\
                           &=&	
e^{-\alpha  N_n  ln( \frac{n }{N_n})}			 .       
\end{eqnarray*}
On the event $\{ \underset{0\leq i \leq n}{\sup} D(0,X_i)
 \leq m  \}$, it comes that :\\
  $$N_n\leq |B_m(\C)| \leq c \ m^d,$$
  and
$$\frac{n}{N_n} \geq \frac{n}{cm^d}.$$
Since  function  $x\mapsto \frac{ln(x)}{x}$ is
decreasing on  $[e, +\infty]$,  if we choose  $m$
such that  
 \begin{eqnarray}\label{m} \frac{n}{cm^d} \geq e ,
 \end{eqnarray}
  then we can write  :
\begin{eqnarray*}
 E_0( \underset{ x;L_{x;n} >0}{\prod }  L_{x;n}^{-\alpha}  \ \ 
  ) 
 &\geq&
e^{-\alpha cm^dln(\frac{n}{cm^d})} \P_0(\underset{0\leq i \leq n}{\sup} |X_i| 
\leq m 
    ).\\
\end{eqnarray*}
Then by using  (\ref{deb}), we deduce :
\begin{eqnarray*}
 E_0( \underset{ x;L_{x;n} >0}{\prod }  L_{x;n}^{-\alpha}  \ \
  ) 
 \geq
e^{-\alpha cm^d ln(\frac{n}{cm^d})}  
e^{-c(m+\frac{n}{m^2}  )} .
\end{eqnarray*}
Taking   $m=(\frac{n}{ln(n)})^{\frac{1}{d+2}}$,
inquality (\ref{m}) is well satisfied for large
enough  $n$.\\
Finaly, for large enough  $n$ we obtain,
$$ E_0( \underset{ x;L_{x;n} >0}{\prod }  L_{x;n}^{-\alpha}   ) 
 \succeq  e^{-n^{\frac{d}{d+2}}  ln(n)^{\frac{2}{d+2} }}  .$$  
 So, 
 \begin{prop} \label{f22}For all $\alpha > 1/2$, $Q$ a.s on
 the set  $|\C|=+\infty$ and for large enough $n$,
 $$ E_0( \underset{ x;L_{x;n} >0}{\prod }  L_{x;n}^{-\alpha}  \ \  ) 
 \succeq  e^{-n^{\frac{d}{d+2}}  ln(n)^{\frac{2}{d+2} }}  . $$
 \end{prop}
 \begin{rem} In the proof of the lower bound, we
 have only used the assumption  that  $\alpha \geq 0$, so this
 bound is valid for all  $\alpha \geq 0.$
\end{rem}
So the second assertion of Theorem \ref{ext}
follows from propositions  \ref{f21} and
\ref{f22}.\\
\\
\\
\bf{  Acknowledgment}\\

 \it{The author would like to thank Pierre Mathieu and Anna Erschler  for usefull remarks.}




%
%
%
%
%
%
%

\end{document}